\documentclass[reqno]{amsart}
\usepackage{amsmath,amssymb,cite}
\usepackage[mathscr]{euscript}
\usepackage{tikz}

\usepackage{xcolor}

\definecolor{bblue}{rgb}{.2,0.2,.8}

\theoremstyle{plain}
\newtheorem{theorem}{Theorem}[section]
\newtheorem{proposition}[theorem]{Proposition}
\newtheorem{lemma}[theorem]{Lemma}
\newtheorem{corollary}[theorem]{Corollary}

\newtheorem{definition}[theorem]{Definition}

\newtheorem{remark}[theorem]{Remark}

\numberwithin{equation}{section}
\numberwithin{theorem}{section}

\newcommand{\mc}[1]{{\mathcal #1}}
\newcommand{\mb}[1]{{\mathbf #1}}
\newcommand{\mf}[1]{{\mathfrak #1}}
\newcommand{\bs}[1]{{\boldsymbol #1}}
\newcommand{\bb}[1]{{\mathbb #1}}
\newcommand{\ms}[1]{{\mathscr #1}}

\renewcommand{\epsilon}{\varepsilon}

\title[Metastability from the large deviations point of view]
{Metastability from the large deviations point of view: A
$\Gamma$-expansion of the level two large deviations rate functional
of non-reversible finite-state Markov chains}

\author{C. Landim}
\address{Claudio Landim
  \hfill\break\indent IMPA \hfill\break\indent Estrada Dona Castorina
  110, \hfill\break\indent
J. Botanico, 22460 Rio de Janeiro, Brazil\hfill\break\indent
  {\normalfont and} \hfill\break\indent CNRS UMR 6085, Universit\'e de
  Rouen, \hfill\break\indent Avenue de l'Universit\'e, BP.12,
  Technop\^ole du Madril\-let, \hfill\break\indent
F76801 Saint-\'Etienne-du-Rouvray, France.} 
\email{landim@impa.br}

\begin{document}

\begin{abstract}
Consider a sequence of continuous-time Markov chains
$(X^{(n)}_t:t\ge 0)$ evolving on a fixed finite state space $V$.  Let
$\ms I_n$ be the level two large deviations rate functional for
$X^{(n)}_t$, as $t\to\infty$.  Under a hypothesis on the jump rates,
we prove that $\ms I_n$ can be written as
$\ms I_n = \ms I^{(0)} \,+\, \sum_{1\le p\le \mf q} (1/\theta^{(p)}_n)
\, \ms I^{(p)}$ for some rate functionals $\ms I^{(p)}$. The weights
$\theta^{(p)}_n$ correspond to the time-scales at which the sequence
of Markov chains $X^{(n)}_t$ exhibit a metastable behavior, and the
zero level sets of the rate functionals $\ms I^{(p)}$ identify the
metastable states.
\end{abstract}

\noindent



\maketitle

\section{Introduction}
\label{sec0}

Fix a finite set $V$ and consider a sequence
$\color{blue} (X^{(n)}_t : t\ge 0)$, $n\ge 1$, of $V$-valued,
irreducible continuous-time Markov chains. Denote the jump rates by
$\color{blue} R_n\colon V \times V \to \bb R_+$, and the generator by
$\ms L_n$, so that
\begin{equation}
\label{47}
{ \color{blue} (\ms L_n f)(x)}
\;=\; \sum_{y \in V} R_n(x,y)\, \{\,
f(y)\,-\, f(x)\,\}\;, \quad f\colon V\to \bb R\;.
\end{equation}
Let $\color{blue} \pi_n$ be the unique stationary state.

Denote by $\color{blue} \ms P(V)$ the space of probability measures on
$V$ endowed with the weak topology, and by $L^{(n)}_t$ the empirical
measure of the chain $X^{(n)}_t$ defined as :
\begin{equation}
\label{48}
{ \color{blue} L^{(n)}_t}
\;:=\; \frac{1}{t} \int_0^t \delta_{X^{(n)}_s}\; ds \;,
\end{equation}
where $\color{blue} \delta_x$, $x\in V$, represents the Dirac measure
concentrated at $x$. Thus, $L^{(n)}_t$ is a random element of
$\color{blue}\ms P(V)$ and $L^{(n)}_t (V_0)$, $V_0 \subset V$, stands
for the average amount of time the process $X^{(n)}_t$ stays at $V_0$
in the time interval $[0,t]$,

As the Markov chain $X^{(n)}_t$ is irreducible, by the ergodic
theorem, for any starting point $x\in V$, as $t\to\infty$, the
empirical measure $L^{(n)}_t$ converges in probability to the stationary
state $\color{blue} \pi_n$.

Donsker and Varadhan \cite{dv75} proved the associated large
deviations principle: for any $x\in V$, $\mu\in \ms P(V)$,
\begin{equation}
\label{46}
\mb P^n_{\! x} \big[\, L^{(n)}_t  \,\sim\, \mu\,\big] \;\approx \;
e^{-t\,  \ms I_n(\mu)}\;, \quad \text{as $t\to\infty$} \;.
\end{equation}
In this formula, $\color{blue} \mb P_{\! x}=\mb P^n_{\! x}$, $x\in V$,
represents the distribution of the process $X^{(n)}_t$ starting from
$x$, and $\ms I_n \colon \ms P(V) \to [0,+\infty)$ be the level two
large deviations rate functional given by
\begin{equation}
\label{35}
{\color{blue}  \ms I_n (\mu) 
\;: =\; \sup_{H} J^{(n)}_H(\mu) } \;:=\;
\sup_{H} \,-\,  \int_{V} 
e^{- H} \, \ms L_n e^{H}\; d\mu\;,
\end{equation}
where the supremum is carried over all functions
$H\colon V \to \bb R$.  A precise statement of \eqref{46} requires
some notation and is postponed to the next section.  The functional
$\ms I_n$ provides the cost for the empirical measure $L^{(n)}_t$ to
be close to $\mu$ for a very large $t$. By Lemma \ref{l32}, as the
process is irreducible, $\ms I_n(\mu)=0$ if, and only if,
$\mu = \pi_n$,

We examine in this article the behavior of the functionals $\ms I_n$
as $n\to\infty$ under some natural hypotheses on the jump rates
$R_n$. Assume, initially, that the jump rates $R_n(x,y)$ converge, as
$n\to\infty$, to a limit represented by $\bb R_0(x,y)$:
\begin{equation}
\label{o-01}
{\color{blue} \bb R_{0} (x,y)}
\;:=\; \lim_n R_n(x,y) \;\in\, \bb R_+ \;, \quad y\,\neq\, x \,\in\, V \;. 
\end{equation}
Denote by $\color{blue} \bb L_0$ the generator associated to these
rates and by $\color{blue} \ms I^{(0)}$ the corresponding large
deviations rate functional.  By Proposition \ref{mt2}, as
$n\to\infty$, $\ms I_n(\mu)$ converges to $\ms I^{(0)}(\mu)$ for all
$\mu\in \ms P(V)$.

If the Markov chain $\color{blue} \bb X_t$ induced by the jump rates
$\bb R_0$ has only one closed irreducible class, the asymptotic
analysis of the functionals $\ms I_n$ ends with Proposition \ref{mt2}.
In contrast, if $\bb X_t$ has more than one closed irreducible class
a finer description of $\ms I_n$ is possible.

Denote by $\color{blue} \ms V_1, \dots, \ms V_{\mf n}$, $\mf n\ge 2$,
the closed irreducible classes of $\bb X_t$. Let
$\color{blue} \pi^\sharp_j$ be the stationary state supported in
$\ms V_j$. By Lemma \ref{l32}, $\ms I^{(0)}$ vanishes at any convex
combination of the measures $\pi^\sharp_j$. Since, by Proposition
\ref{mt2}, $\ms I_n(\mu)$ converges to $\ms I^{(0)}(\mu)$ it is
natural to consider the sequence $\theta_n\, \ms I_n(\mu)$, for some
$\theta_n\to\infty$ and a convex combination
$\mu = \sum_j \omega_j\, \pi^\sharp_j$ of the measures $\pi^\sharp_j$,
longing to obtain a non-trivial limit.

To find the correct sequence $\theta_n$, remark that, by \eqref{35},
$\theta_n\, \ms I_n$ represents the large deviations rate functional
of the Markov chain induced by the generator $\theta_n\, \ms L_n$,
that is, the rate functional of the Markov chain $X^{(n)}_{t}$
observed at the time scale $\theta_n$:
$\color{blue} X^{1,n}_t := X^{(n)}_{\theta_n\, t}$.

Denote by $\beta_{n,j}$ the transition time from $\ms V_j$ to
$\cup_{k\not = j} \ms V_k$, this is the mean time for the process
$X^{(n)}_t$ to hit $\cup_{k\not = j} \ms V_k$ when it starts from
$\ms V_j$. For the sake of the argument, assume that
$\beta_n = \beta_{n,1}$.

Fix a time-scale $\theta_n$ such that $\theta_n \to\infty$,
$\theta_n/\beta_n\to 0$. Denote this last relation by
$\theta_n \prec \beta_n$ or $\beta_n\succ \theta_n$. As the transition
time from $\ms V_j$ to $\cup_{k\not = j} \ms V_k$ is of order
$\beta_n$ and $\beta_n \succ \theta_n$, in the time-scale $\theta_n$
starting from $\ms V_j$ the chain $X^{(n)}_t$ does not visit the set
$\cup_{k\not = j} \ms V_k$. Therefore, the cost for keeping the
process at $\ms V_j$ should vanish, and one expects
$\theta_n\, \ms I_n (\pi^\sharp_1) \to 0$. Actually, as
$\beta_{n,j} \succ \theta_n$ for all $j$, the same conclusion should
hold for all measures $\pi^\sharp_j$, and to derive a non-trivial
limit for $\theta_n\, \ms I_n$ one has to observe the chain
$X^{(n)}_t$ in a time-scale at least of the order $\beta_n$.

In the time scale $\beta_n$, starting from $\ms V_j$
the process visits $\cup_{k\not = j} \ms V_k$. There is, in
consequence, a positive cost to maintain it at $\ms V_j$ and
$\theta_n\, \ms I_n (\pi^\sharp_1)$ should converge to a positive
limit. If all sequence $\beta_{n,j}$ are of the same order, this
completes the description of $\ms I_n$. Otherwise, one has to go to
longer time-scales.

The main result of this article, Theorem \ref{mt1}, presents the
time-scales $1\prec \theta^{(1)}_n \prec \cdots \prec
\theta^{(\mf q) }_n$ and functionals $\ms I^{(1)},  \cdots ,
\ms I^{(\mf q)}_n$ such that
\begin{equation*}
\theta^{(p)}_n\, \ms I_n \;\longrightarrow \ms I^{(p)}\;, \quad 1\le
p\le \mf q\;.
\end{equation*}
This result permits to write the functional $\ms I_n$ as the expansion
\begin{equation}
\label{60}
\ms I_n  \;=\; \ms I^{(0)} \;+\; \sum_{p=1}^{\mf q}
\frac{1}{\theta^{(p)}_n} \, \ms I^{(p)}\;.
\end{equation}
The weights $\theta^{(p)}_n$ correspond to the time-scales at which
the sequence of Markov chains $X^{(n)}_t$ exhibit a metastable
behavior, and the zero level sets of the rate functionals
$\ms I^{(p)}$ identify the metastable states.

The proof o Theorem \ref{mt1} relies on \cite{bl7, lx16} where the
metastable behavior of the sequence $X^{(n)}_t$ has been
investigated. The expansion \eqref{60} has been derived for reversible
diffusions in \cite{GM17} and for reversible finite state Msrkov
chains in \cite{bgl2}.  It should be a universal property of Markov
chains and should hold for dynamics whose state space depend on $n$
and which exhibit a metastable behavior at different time-scales. This
includes, among others models, randoms walks and diffusions in
potential fields \cite{s95, begk02, lmt2015, ls2018, ls19, lms19,
ls22, ls22b, rs18}, condensing zero-range processes \cite{bl3, agl,
l2014, s2018, or19}, inclusion processes \cite{bdg17, GRV-13, CCG-14,
kim21, ks21}.

We believe that the argument proposed here to derive the expansion of
the large deviations rate functional cab be adapted to cover these
dynamics. 

\section{Notation and Results}
\label{sec1}

We present in this section the main result of the article.  Consider a
sequence $\color{blue} (X^{(n)}_t : t\ge 0)$, $n\ge 1$, of $V$-valued,
irreducible continuous-time Markov chains whose generator is given by
\eqref{47}. 

Denote by $\color{blue} D(\bb R_+, W)$, $W$ a finite set, the space of
right-continuous functions $\mf x: \bb R_+ \to W$ with left-limits
endowed with the Skorohod topology and the associated Borel
$\sigma$-algebra. Let $\color{blue} \mb P_{\! x}=\mb P^n_{\! x}$,
$x\in V$, be the distribution of the process $X^{(n)}_t$ starting from
$x$. This is the probability measure on the path space $D(\bb R_+, V)$
induced by the Markov chain $X^{(n)}_t$ starting from $x$. Expectation
with respect to $\mb P_{\! x}$ is represented by
$\color{blue} \mb E_x$.

Recall the definition of the empirical measure $L^{(n)}_t$ introduced
in \eqref{48}.  Donsker and Varadhan \cite{dv75} proved a large
deviations principle for the empirical measure $L^{(n)}_t$. More
precisely, they showed that for any subset $\ms A$ of $\ms P(V)$,
\begin{equation}
\label{50}
\begin{aligned}
-\, \inf_{\mu\in \ms A^o} \ms I_n (\mu) \;&\le\; \liminf_{t\to \infty}
\inf_{x\in V}\, \frac{1}{t}\,
\ln \mb P^{(n)}_x \big[ \, L^{(n)}_t \in \ms A\,\big] \\
&\qquad \;\le\; \limsup_{t \to \infty} \sup_{x\in V}\,
\frac{1}{t}\, \ln \mb P^{(n)}_x \big[ \, L^{(n)}_t \in \ms A\,\big]
\;\le\; -\, \inf_{\mu\in \overline{\ms A}} \ms I_n (\mu)\;.
\end{aligned}
\end{equation}
In this formula, $\ms A^o$, $\overline{\ms A}$ represent the interior,
closure of $\ms A$, respectively, and $\ms I_n$ is the large
deviations rate functional introduced in \eqref{35}.

We examine in this article the asymptotic behavior of the rate
functional $\ms I_n$.  In the context of large deviations, the
appropriate notion of convergence is the $\Gamma$-convergence defined
as follows.  We refer to \cite{Br} for an overview on this subject.

Fix a Polish space $\mc X$ and a sequence $(U_n : n\in\bb N)$ of
functionals on $\mc X$, $U_n\colon \mc X \to [0,+\infty]$.  The
sequence $U_n$ \emph{$\Gamma$-converges} to the functional
$U\colon \mc X\to [0,+\infty]$ if and only if the two following
conditions are met:

\begin{itemize}
\item [(i)]\emph{$\Gamma$-liminf.} The functional $U$ is a
$\Gamma$-liminf for the sequence $U_n$: For each $x\in\mc X$ and each
sequence $x_n\to x$, we have that $\liminf_n U_n(x_n) \ge U(x)$.

\item [(ii)]\emph{$\Gamma$-limsup.} The functional $U$ is a
$\Gamma$-limsup for the sequence $U_n$: For each $x\in\mc X$ there
exists a sequence $x_n\to x$ such that
\begin{equation}
\label{30}
\limsup_{n\to \infty} U_n(x_n) \;\le\; U(x)\;.
\end{equation}
\end{itemize}

Recall that we denote by $R_n(x,y)$ the jump rates of the Markov chain
$X^{(n)}_t$. Assume that the rates converge, as $n\to \infty$, to a
finite limit denoted by $\bb R_0(x,y)$, see \eqref{o-01}, and that
$\bb R_0(x',y') >0$ for some $y'\not = x'\in V$.  The jump rates
$\bb R_0 (x,y)$ induce a continuous-time Markov chain on $V$, denoted
by $\color{blue} (\bb X_t:t\ge 0)$, which, of course, may be
reducible. Denote by $\color{blue} \bb L^{(0)}$ its generator and by
$\ms I^{(0)} : \ms P(V) \to \bb R_+$ the associated occupation-time
large deviations rate functional, given by
\begin{equation}
\label{o-f11}
{\color{blue} \ms I^{(0)} (\mu)} \;=\; 
\sup_{H}   \, -\, \sum_{x\in V} e^{-H(x)} \, \big[\,
(\bb L^{(0)} e^{H}) \, (x) \;\big]\;  \mu(x)  \;,
\end{equation}
where the supremum is carried over all functions
$H: V \to \bb R$. Next result is proved in Section \ref{sec4}. 

\begin{proposition}
\label{mt2}
The sequence of functionals $\ms I_n$ $\Gamma$-converges to
$\ms I^{(0)}$.
\end{proposition}

Assume from now on that
\begin{equation}
\label{49}
\text{the Markov chain $\bb X_t$ has more than one closed
irreducible class.}
\end{equation}
Under this hypothesis we may investigate further the asymptotic
behavior of the rate functional $\ms I_n$.

\subsection*{Main assumption}

To examine the convergence of $\theta_n\, \ms I_n$ for some sequence
$\theta_n\to \infty$, we introduce a natural hypothesis on the jump
rates proposed in \cite{bl7} and adopted in \cite{lx16, fk17, bgl2}.

For two sequences of positive real numbers $(\alpha_n : n\ge 1)$,
$(\beta_n : n\ge 1)$, $\color{blue} \alpha_n \prec \beta_n$ or
$\color{blue} \beta_n \succ \alpha_n$ means that
$\lim_{n\to\infty} \alpha_n/\beta_n = 0$. Similarly,
$\color{blue} \alpha_n \preceq \beta_n$ or
$\color{blue} \beta_n \succeq \alpha_n$ indicates that either
$\alpha_n \prec \beta_n$ or $\alpha_n/\beta_n$ converges to a positive
real number $a\in (0,\infty)$.

Two sequences of positive real numbers $(\alpha_n : n\ge 1)$,
$(\beta_n : n\ge 1)$ are said to be \emph{comparable} if
$\alpha_n \prec \beta_n$, $\beta_n \prec \alpha_n$ or
$\alpha_n / \beta_n \to a \in (0,\infty)$. This condition excludes the
possibility that
$\liminf_n \alpha_n /\beta_n \neq \limsup_n \alpha_n /\beta_n$.

A set of sequences
$(\alpha^{\mf u}_n: n\ge 1)$, $\mf u \in \mf R$, of positive real
numbers, indexed by some finite set $\mf R$, is said to be comparable
if for all $\mf u$, $\mf v \in\mf R$ the sequence
$(\alpha^{\mf u}_n : n \ge 1)$, $(\alpha^{\mf v}_n : n \ge 1)$ are
comparable.

Denote by $\color{blue} E \subset \{(x,y) \in V\times V: y\not = x\}$
a set of directed edges, and assume that for all $n\ge 1$,
\begin{equation}
\label{51}
\text{$R_n(x,y)>0$ if, and only if, $(x,y)\in E$}\;.
\end{equation}
Let $\bb Z_+ = \{0, 1, 2, \dots \}$, and $\color{blue} \Sigma_m$,
$m\ge 1$, be the set of functions $k: E \to \bb Z_+$ such that
$\sum_{(x,y)\in E} k(x,y) = m$.  We assume, hereafter, that for every
$m\ge 1$ the set of sequences
\begin{equation}
\label{mh}
\big(\, \prod_{(x,y)\in E} R_n(x,y)^{k(x,y)} : n \ge 1 \,\big)
\;,\quad k\in\Sigma_m \;,
\end{equation}
is comparable.

As observed in \cite{bl7} (see Remark 2.2 in \cite{bgl2}), assumption
\eqref{mh} is fulfilled by all statistical mechanics models which
evolve on a fixed finite state space and whose metastable behaviour
has been derived.

\subsection*{Tree decomposition}

Under the assumptions \eqref{49}, \eqref{51} and \eqref{mh},
\cite{bl7, lx16} constructed a rooted tree which describes the
behaviour of the Markov chain $X^{(n)}_t$ at all different
time-scales. We refer to \cite{bgl2} for a clear presentation of the
construction as well as a simple example, and recall here the main
ideas. 

Denote by $\mf q +1 \ge 2$ the number of generations of the tree.  The
elements of the $p$-th generation form a partition of $V$, and are
represented by
$\color{blue} \ms W^{(p)}_1, \dots, \ms W^{(p)}_{\mf m_p}, \Omega_p$
for some finite increasing sequence
$\color{blue} 1 = \mf m_1 \le \cdots \le \mf m_{\mf q+1}$.  The set
$\Omega_p$ may be empty while the sets $\ms W^{(p)}_j$ are all
non-empty. As $\mf m_p\ge 1$, each generation has at least one
element. Here is a list of the main properties of the tree:

\begin{itemize}
\item[(1.a)] Each generation of the tree forms a partition of $V$;

\item[(1.b)] The root, or $0$-th generation, is the set $V$. The first
generation has one or two elements depending on whether $\Omega_1$ is
empty or not. If $\Omega_1 = \varnothing$, it has one element, the set
$\ms W^{(1)}_1 = V$. If $\Omega_1 \neq \varnothing$, it has two
elements, the sets $\ms W^{(1)}_1$ and $\Omega_1 = (\ms W^{(1)}_1)^c$.

\item[(1.c)] Each child of a vertex is a subset of its parent: For
each $0\le p \le \mf q$, $1\le j\le \mf m_{p+1}$, either
$\ms W^{(p+1)}_j \subset \ms W^{(p)}_k$ for some $1\le k\le \mf m_{p}$
or $\ms W^{(p+1)}_j \subset \Omega_p$. Moreover,
$\Omega_{p+1}\subset \Omega_p$;

\item[(1.d)] According to the notation, the number of elements of
generation $p$ is equal to
$\mf m_p + \bs 1_{\Omega_p \neq \varnothing}$, where
$\bs 1_{\ms A\neq \varnothing}$ is equal to $1$ if $\ms A$ is not
empty and $0$ otherwise. Starting from the first generation, the
number of descendents of a generation strictly increases: for
$1\le p \le \mf q$,
$\mf m_p + \bs 1_{\Omega_p \neq \varnothing} < \mf m_{p+1} + \bs
1_{\Omega_{p+1} \neq \varnothing}$.
\end{itemize}

\subsection*{Construction of the tree}

We describe in this subsection the details of the construction of the
tree.  it is formed from the leaves to the root.  The leaves
$\ms V^{(1)}_1, \dots, \ms V^{(1)}_{\mf n_1}$ are the closed
irreducible classes of the Markov chain $\bb X_t$ introduced in the
previous section, and $\Delta_1$ the set of transient states. The sets
$\ms V^{(1)}_j$ were represented there by $\ms V_j$, a notation
frequently adopted below.  In view of the definition of the sets
$\ms V^{(1)}_j$, $\mf n_1$ corresponds to the number of closed
irreducible classes of the process $\bb X_t$, which we assumed in
\eqref{49} to be larger than or equal to $2$.

We turn to the construction of the parents of the leaves. This
procedure will be repeated recursively to define all generations from
the leaves to the root.  Denote by $\Phi_1: V \to S_1$ the projection
defined by
\begin{equation*}
\Phi_1(\,\cdot\,) \;=\; \sum_{j\in S_1} j\, \chi_{\ms V^{(1)}_j}
(\,\cdot\,)\;, 
\end{equation*}
where $\color{blue} \chi_{\ms A}$ stands for the indicator function of
the set $\ms A$. Hence, $\Phi_1$ projects to $0$ all elements of
$\Delta_1$ and to $j$ the ones of $\ms V^{(1)}_j$.

It follows from the main results of \cite{lx16, llm18} and
\cite[Lemmata 4.7]{bgl2} that there exist a time-scale
$\theta^{(1)}_n\succ 1$ and a $S_1$-valued Markov chain
$\color{blue} \bb X^{(1)}_t$ (note that this process does not take the
value $0$), such that the finite-dimensional distributions of
$\Phi_1( X^{(n)}_{t\theta^{(1)}_n})$ converge to those of $\bb
X^{(1)}_1$: 
\begin{equation}
\label{58}
\Phi_1\big(\, X^{(n)}_{t\theta^{(1)}_n}\,\big) \;
\overset{\rm f. d. d.}{-\!-\!\!\!\longrightarrow}
\; \bb X^{(1)}_t \;.
\end{equation} 

Denote by $\color{blue} \mf R^{(1)}_1, \dots, \mf R^{(1)}_{\mf n_2}$
the recurrent classes of the $S_1$-valued Markov chain
$\bb X^{(1)}_t$, and by $\color{blue} \mf T_1$ the transient
states. Let ${\color{blue} \mf R^{(1)}} = \cup_j \mf R^{(1)}_j$, and
observe that $\{\mf R^{(1)}_1, \dots $,
$\mf R^{(1)}_{\mf n_{2}}, \mf T_1\}$ forms a partition of the set
$S_1$. This partition of $S_1$ induces a new partition of $V$. Let
\begin{equation}
\label{59}
{\color{blue}  \ms V^{(2)}_m}
\;:= \; \bigcup_{j\in \mf R^{(1)}_m} \ms V^{(1)}_j\;, \quad
{\color{blue}  \ms T^{(1)}}
\;:= \; \bigcup_{j\in \mf T_1} \ms V^{(1)}_j\;, 
\quad m\in  {\color{blue} S_{2} \;:=\; \{1, \dots, \mf n_{2}\}} \;,
\end{equation}
so that $V  \,=\,  \Delta_{2}   \, \cup \, \ms V^{(2)}$, where
\begin{equation*}
{\color{blue} \ms V^{(2)}}
\;=\; \bigcup_{m\in S_{2}} \ms V^{(2)}_{m}\;,
\quad
{\color{blue}  \Delta_{2}}
\;:= \; \Delta_1 \,\cup\, \ms T^{(2)} \;.
\end{equation*}

It is shown in \cite{lx16} that the Markov chain $\bb X^{(1)}_t$ is
non-degenerate in the sense that there exists at least one edge
$(j,k)$, $k\neq j\in S_1$, such that $r^{(1)}(j,k)>0$, where
$r^{(1)}(\,\cdot\,,\,\cdot\,)$ represents the jump rates of the Markov
chain $\bb X^{(1)}_t$. In particular, either $j$ is a transient state
or $j$ and $k$ belong to the same closed irreducible class. Therefore,
the number of recurrent classes ($\mf n_2$) is strictly smaller than
the number of $S_1$ elements ($\mf n_1$): $\mf n_2<\mf n_1$. Since, on
the other hand, $\Delta_{2} \supset \Delta_1$, the number of leaves'
parents (the generation $\mf q-1$ in the previous subsection) is
strictly smaller than the one of leaves (the generation $\mf q$).

In conclusion, from the partition
$\ms V^{(1)}_1, \dots, \ms V^{(1)}_{\mf n_1}, \Delta_1$, the theory
presented in \cite{lx16} produced a time-scale
$\theta^{(1)}_n \succ 1$, a $S_1$-valued Markov chain $\bb X^{(1)}_t$,
and a coarser partition $\ms V^{(2)}_1, \dots, \ms V^{(2)}_{\mf n_2}$,
$\Delta_2$. The construction of the tree proceeds by recurrence.

Assume that, for some $p> 1$,
the recursion has produced
\begin{itemize}
\item[(a)] Time scales
$1\prec \theta^{(1)}_n\prec \cdots \prec \theta^{(p-1)}_n$;

\item[(b)] $S_q$-valued Markov chains $\bb X^{(q)}_t$, $1\le q<p$, where
$\color{blue} S_q = \{1, \dots, \mf n_q\}$; 

\item[(c)] Partitions
$\ms V^{(r)}_1, \dots, \ms V^{(r)}_{\mf n_r}, \Delta_r$,
$1\le r\le p$
\end{itemize}
satisfying \eqref{58}, \eqref{59} (with the obvious modifications
which appear in \eqref{58b}, \eqref{59b}). Assume, furthermore, that
$\mf n_p>1$. Then, by \cite{lx16, llm18} and \cite[Lemmata 5.6]{bgl2},
there exist a time-scale $\theta^{(p)}_n\succ \theta^{(p-1)}_n$ and a
$S_p$-valued Markov chain $\color{blue} \bb X^{(p)}_t$ such that the
finite-dimensional distributions of
$\Phi_p( X^{(n)}_{t\theta^{(p)}_n})$ converge to those of
$\bb X^{(p)}_1$:
\begin{equation}
\label{58b}
\Phi_p\big(\, X^{(n)}_{t\theta^{(p)}_n}\,\big) \;
\overset{\rm f. d. d.}{-\!-\!\!\!\longrightarrow}
\; \bb X^{(p)}_t \;.
\end{equation}
In this formula, $\Phi_p: V \to S_p$ represents the projection
defined by
\begin{equation*}
\Phi_p(\,\cdot\,) \;=\; \sum_{j\in S_p} j\, \chi_{\ms V^{(p)}_j}
(\,\cdot\,)\;.
\end{equation*}

Denote by $\color{blue} \mf R^{(p)}_1, \dots, \mf R^{(p)}_{\mf n_{p+1}}$
the recurrent classes of the $S_p$-valued Markov chain
$\bb X^{(p)}_t$, and by $\color{blue} \mf T_p$ the transient
states. Let ${\color{blue} \mf R^{(p)}} = \cup_j \mf R^{(p)}_j$, and
observe that $\{\mf R^{(p)}_1, \dots $,
$\mf R^{(p)}_{\mf n_{p+1}}, \mf T_p\}$ forms a partition of the set
$S_p$. This partition of $S_p$ induces a new partition of $V$. Let
\begin{equation}
\label{59b}
{\color{blue}  \ms V^{(p+1)}_m}
\;:= \; \bigcup_{j\in \mf R^{(p)}_m} \ms V^{(p)}_j\;, \quad
{\color{blue}  \ms T^{(p)}}
\;:= \; \bigcup_{j\in \mf T_p} \ms V^{(p)}_j\;, 
\quad m\in  {\color{blue} S_{p+1} \;:=\; \{1, \dots, \mf n_{p+1}\}} \;,
\end{equation}
so that $V  \,=\,  \Delta_{p+1}   \, \cup \, \ms V^{(p+1)}$, where
\begin{equation*}
{\color{blue} \ms V^{(p+1)}}
\;=\; \bigcup_{m\in S_{p+1}} \ms V^{(p+1)}_{m}\;,
\quad
{\color{blue}  \Delta_{p+1}}
\;:= \; \Delta_p \,\cup\, \ms T^{(p+1)} \;.
\end{equation*}

As above, it is shown in \cite{lx16} that the Markov chain
$\bb X^{(p)}_t$ is non-degenerate so that $\mf n_{p+1}<\mf n_p$.  The
induction can proceed if $\mf n_{p+1}>1$, otherwise it ends. Denote by
$\color{blue} \mf q$ the first integer $r$ such that $\mf n_{r+1}=1$,
(equivalently, the first $r$ such that the Markov chain
$\bb X^{(r)}_t$ has only one recurrent class). At this point the
iteration stops and the partition of $V$ produced is
$\{\ms V^{(\mf q+1)}_1, \Delta_{\mf q+1}\}$ which may have one or two
elements, depending on whether $\Delta_{\mf q+1}$ is empty or not.

To recover the tree presented in the previous subsection, add a final
partition equal to $V$ which will identified to the root of the tree,
and for $1\le p \le \mf q +1$,
$k\in S_{\mf q+2 -p} = \{1, \dots, \mf n_{\mf q+2 -p}\}$, set
\begin{equation*}
\mf m_p \;:=\; \mf n_{\mf q+2 -p}  \;,
\quad
\ms W^{(p)}_k \; :=\; \ms V^{({\mf q+2 -p})}_k  \;,
\quad
\Omega_p \;=\;  \Delta_{\mf q +2- p} \;.
\end{equation*}
It is easy to check that conditions (1.a)--(1.d) are fulfilled. 

\subsection*{A set of measures}

We construct in this subsection a set of probability measures
$\pi^{(p)}_j$, $1\le p\le \mf q +1$, $j\in S_p$, on $V$ which
describe the evolution of the chain $X^{(n)}_t$ and such that
\begin{equation}
\label{o-52}
\text{ the support of $\pi^{(p)}_j$ is the set $\ms V^{(p)}_j$}\; .
\end{equation}

We proceed by induction. Let $\pi^{(1)}_j$, $j\in S_{1}$, be the
probability measure on $\ms V^{(1)}_j$ given by
$\color{blue} \pi^{(1)}_j = \pi^\sharp_j$, where, recall,
$\pi^\sharp_j$ represents the stationary states of the Markov chain
$\bb X_t$ restricted to the closed irreducible class
$\ms V^{(1)}_j = \ms V_j$. Clearly, condition \eqref{o-52} is
fulfilled.

Fix $1\le p\le \mf q$, and assume that the probability measures
$\pi^{(p)}_j$, $j\in S_p$, has been defined and satisfy condition
\eqref{o-52}. Denote by $\color{blue} M^{(p)}_m(\cdot)$,
$m\in S_{p+1}$, the stationary state of the Markov chain
$\bb X^{(p)}_t$ restricted to the closed irreducible class
$\mf R^{(p)}_m$. The measure $ M^{(p)}_m$ is understood as a measure
on $S_p=\{1, \dots, \mf n_p\}$ which vanishes on the complement of
$\mf R^{(p)}_m$.  Let $\pi^{(p+1)}_m$ be the probability measure on
$\ms V^{(p)}_m$ given by
\begin{equation}
\label{o-80}
{\color{blue} \pi^{(p+1)}_m (x)}
\;:=\; \sum_{j\in \mf R^{(p)}_m} M^{(p)}_m(j)\, \pi^{(p)}_j (x)\;, \quad
x\in V\;.
\end{equation}

Clearly, condition \eqref{o-52} holds, and the measure
$\pi^{(p+1)}_m$, $1\le p\le \mf q$, $m\in S_{p+1}$, is a convex
combination of the measures $\pi^{(p)}_j$, $j\in \mf R^{(p)}_m$.
Moreover, by \cite[Theorem 3.1 and Proposition 3.2]{bgl2}, for all
$z\in \ms V^{(p)}_j$,
\begin{equation}
\label{o-58}
\lim_{n\to\infty} \frac{\pi_n(z)}{\pi_n(\ms V^{(p)}_j)}
\;=\; \pi^{(p)}_j(z) \,\in\, (0,1] \;, \quad
\lim_{n\to \infty} \pi_n(\Delta_{\mf q+1}) \;=\; 0
\;.
\end{equation}

By \eqref{o-80}, the measures $\pi^{(p)}_j$, $2\le p\le \mf q+1$,
$j\in S_p$, are convex combinations of the measures $\pi^{(1)}_k$,
$k\in S_1$. By \eqref{o-58}, for all $x\,\in\, \ms V^{(\mf q+1)}$,
$\lim_{n\to\infty} \pi_n(x)$ exists and belongs to $(0,1]$. By
\eqref{o-58}, and since by (1.c) $\Delta_p \subset \Delta_{p+1}$ for
$1\le p\le \mf q$, $\lim_{n\to \infty} \pi_n(\Delta_{p}) \;=\; 0$ for
all $p$.

\subsection*{A complete description of the chain $X^{(n)}_t$}

The statement of next result requires some notation.  Denote by
$H_{\ms A}$, $H^+_{\ms A}$, ${\ms A}\subset V$, the hitting and return
time of ${\ms A}$:
\begin{equation} 
\label{o-201}
{\color{blue} H_{\ms A} } \;: =\;
\inf \big \{t>0 : X^{(n)}_t \in {\ms A} \big\}\;,
\quad
{\color{blue} H^+_{\ms A}} \;: =\;
\inf \big \{t>\tau_1 : X^{(n)}_t \in {\ms A} \big\}\; ,  
\end{equation}
where $\tau_1$ represents the time of the first jump of the chain
$X^{(n)}_t$:
$\color{blue} \tau_1 = \inf\{t>0 : X^{(n)}_t \not = X^{(n)}_0\}$.

For $1\le p\le \mf q+1$, $k\in S_p$, let
\begin{equation*}
\breve{\ms V}^{(p)}_k
\;:= \; \bigcup_{j\in S_{p}\setminus \{k\}} \ms V^{(p)}_{j} \;.
\end{equation*}
Define $\color{blue} \mf a^{(p-1)}\colon V \times S_p \to [0,1]$ as
follows. Fix $j\in S_p$. If $x\not\in \ms V^{(p)}$, set
\begin{equation*}
\mf a^{(p-1)} (x,j) \;: =\; \lim_{n\to\infty} \mb P^n_x
\big[\, H_{\ms V^{(p)}_j} \,<\,  H_{\breve{\ms V}^{(p)}_j}\,\big]\;,
\end{equation*}
while, if $x\in \ms V^{(p)}_k$ for $k\in S_p$, set
$\mf a^{(p-1)} (x,j)=\delta_{j,k}$. For $p=\mf q+1$, as $S_{\mf q+1}$
is a singleton, $\mf a^{(\mf q)} (x,1) = 1$ for all $x\in V$.

Denote by $p^{(n)}_t(x,y)$ the transition probability of
the Markov chain $X^{(n)}_t$:
\begin{equation*}
{ \color{blue}  p^{(n)}_t(x,y)}
\;:=\; \mb P^n_{\! x} \big[\, X_t \,=\, y\,\big]\;,
\quad x\,,\, y \in V\;,\;\; t\,>\,0\;.
\end{equation*}
Next result is \cite[Theorem 3.1 and Proposition 3.2]{bgl2}.

\begin{theorem}
\label{t1}
Under the hypotheses \eqref{51} and \eqref{mh}, for each
$1\le p\le \mf q$, $t>0$, $x \in V$,
\begin{equation}
\label{o-60}
\lim_{n\to\infty} p^{(n)}_{t \theta^{(p)}_n} (x,\,\cdot\,) \;=\;
\sum_{j\in S_p} \omega^{(p)}_t(x,j)\, \pi^{(p)}_j (\,\cdot\,)\;,
\end{equation}
where
\begin{equation*}
\omega^{(p)}_t(x,j) \;=\;
\sum_{k\in S_{p}} \mf a^{(p-1)} (x,k) \; p^{(p)}_t(k,j) \;,
\end{equation*}
and $p^{(p)}_t(k,j)$ is the transition matrix of the Markov chain $\bb
X^{(p)}_t$.  Moreover, 
\begin{itemize}
\item[(3.a)] Let $\theta^{(0)}_n =1$,
$\theta^{(\mf q +1)}_n =+\infty$ for all $n\ge 1$. For
each $1\le p\le \mf q +1$, sequence $(\beta_n:n\ge 1)$ such that
$\theta^{(p-1)}_n \,\prec\, \beta_n\,\prec\, \theta^{(p)}_n$, and
$x \in V$,
\begin{equation*}
\lim_{n\to\infty} p^{(n)}_{\beta_n} (x,\,\cdot\, ) \;=\;
\;=\; \sum_{j\in S_p}
\mf a^{(p-1)} (x,j)\, \pi^{(p)}_j (\,\cdot\, )\;.
\end{equation*}

\item[(3.b)] For all $1\le p\le \mf q$, $1\le j\le \mf n_p$, $x\in V$, 
\begin{equation*}
\lim_{t\to\infty} \lim_{n\to\infty} p^{(n)}_{t \theta^{(p)}_n} (x,\,\cdot\,)
\;=\; \sum_{m\in S_{p+1}}
\mf a^{(p)}(x,m)\, \pi^{(p+1)}_m (\,\cdot\,)\;.
\end{equation*}
\end{itemize}
\end{theorem}

Equation \eqref{o-60} and properties (3.a), (3.b) describe the
behavior of the Markov chain $X^{(n)}_t$ in all time-scales. By
\eqref{o-60}, for instance, starting from $x$, as $n\to\infty$, the
distribution of $X^{(n)}_{t\theta^{(p)}_n}$ is a convex combination
of the measures $\pi^{(p)}_j$. The weights $\omega^{(p)}_t(x,j)$ have
a simple interpretation: $\omega^{(p)}_t(x,j)$ is equal to the
probability that starting from $x$ the chain reaches the set
$\ms V^{(p)}$ at $\ms V^{(p)}_k$ times the probability that the Markov
chain $\bb X^{(p)}_t$ starting from $k$ is at $j$ at time $t$.

Clearly, by \eqref{o-60}, for all $1\le p\le \mf q$, $j\in S_p$,
$x\in V$,
\begin{equation*}
\lim_{t\to0} \lim_{n\to\infty} p^{(n)}_{t \theta^{(p)}_n} (x,\,\cdot\,) 
\;=\; \sum_{j\in S_p}
\mf a^{(p-1)} (x,j)\, \pi^{(p)}_j (\,\cdot\, )\;.
\end{equation*}

\subsection*{The $\Gamma$-expansion of the rate functional $\ms I_n$}

We are now in a position to state the main result of this article.
Denote by $\color{blue} \ms P(S_p)$, $1\le p\le \mf q$, the set of
probability measures on $S_p$ and by $\color{blue}\bb L^{(p)}$ the
generator of the $S_p$-valued Markov chain $\bb X^{(p)}_t$. Let
$\bb I^{(p)} \colon \ms P (S_p) \to [0,+\infty)$ be the level two
large deviations rate functional of $\bb X^{(p)}_t$ given by
\begin{equation}
\label{40}
{\color{blue} \bb I^{(p)} (\omega) } \, :=\,
\sup_{\mb h} \,-\,  \sum_{j\in S_p} \omega_j \,
e^{- \mb h (j) } \, (\bb L^{(p)} e^{ \mb h})(j)  \;,
\end{equation}
where the supremum is carried over all functions
$\mb h:S_p \to \bb R$.  Denote by
$\ms I^{(p)} \colon \ms P(V) \to [0,+\infty]$ the functional given by
\begin{equation}
\label{o-83b}
{\color{blue} \ms I^{(p)} (\mu) } \, :=\,
\left\{
\begin{aligned}
& \bb I^{(p)} (\omega)   \quad \text{if}\;\;
\mu = \sum_{j\in S_p} \omega_j \, \pi^{(p)}_j \;\; \text{for}\;\;
\omega \in \ms P (S_p)\;,  \\
& +\infty \quad\text{otherwise}\;.
\end{aligned}
\right.
\end{equation}

The main result of the article reads as follows.

\begin{theorem}
\label{mt1}
For each $1\le p\le \mf q$, the functional $\theta^{(p)}_n\, \ms I_n$
$\Gamma$-converges to $\ms I^{(p)}$.
\end{theorem}

This theorem provides an expansion of the large deviations rate
function $\ms I_n$ which can be written as
\begin{equation}
\label{f05}
\ms I_n \;=\; \ms I^{(0)} \;+\; \sum_{p=1}^{\mf q}
\frac{1}{\theta^{(p)}_n} \; \ms I^{(p)}\;.
\end{equation}
Therefore, the rate function $\ms I_n$ encodes all the characteristics
of the metastable behavior of the chain $X^{(n)}_t$. The time-scales
$\theta^{(p)}_n$ appear as the weights of the expansion, and, by
\eqref{o-83b}, the meta-stable states $\pi^{(p)}_j$, $j\in S_p$,
generate the space where the rate functional $\ms I^{(p)} (\mu)$ is
finite.

Next result is a simple consequence of the level two large deviations
principle \eqref{50} and the $\Gamma$-convergence stated in the
previous theorem and in Proposition \ref{mt2}.  (cf. Corollary 4.3 in
\cite{mar}).

\begin{corollary}
\label{cor2}
Fix $0\le p\le \mf q$ and recall that $\theta^{(0)}_n=1$.
For every closed subset $F$ and open subset $G$ of $\ms P(V)$,
\begin{equation*}
\begin{gathered}
\limsup_{n\to \infty} \, \limsup_{t\to \infty}
\, \frac{\theta^{(p)}_n}{t} \,\sup_{x\in V} \, \log\, 
\mb P_{\! x}^n \Big[\, \frac{1}{t} \int_0^t \delta_{X^n_s}\; ds
\,\in\, F \, \Big] \;\le\; -\, \inf_{\mu\in F} \ms I^{(p)} (\mu) \;,
\\
\liminf_{n\to \infty} \, \liminf_{t\to \infty}
\, \frac{\theta^{(p)}_n}{t}  \, \inf_{x\in V} \,\log\, 
\mb P^n_{\! x} \Big[\, \frac{1}{t} \int_0^t \delta_{X^n_s}\; ds
\,\in\, G \, \Big] \;\ge\; -\, \inf_{\mu\in G} \ms I^{(p)} (\mu) \;.
\end{gathered}
\end{equation*}
\end{corollary}

\subsection*{Organisation of the paper}
The article is organised as follows. In Section \ref{sec2}, we obtain
some estimates on the jump rates. This is a technical section which
can be skipped in a first reading. In Section \ref{sec3}, we prove the
$\Gamma-\limsup$ for the sequence of large deviations rate functionals
associated to the trace process. Proposition \ref{mt2} and Theorem
\ref{mt1} are proved in Section \ref{sec4}.

In the appendices we present general results on finite state Markov
chains needed in the proof of Theorem \ref{mt1} and which do not
require assumption \eqref{mh}.  In Appendix \ref{sec01} we derive some
properties of level two large deviations rate functionals. This leads
us to introduce reflected and tilted dynamics. In Appendix
\ref{sec-a2} we investigate the convergence of these functionals, and
in Appendix \ref{sec-a3} the relation between the trace process and
the rate functionals. Throughout the article we assume the reader to
be familiar with the results presented in the appendices.

\section{The jump rates}
\label{sec2}

In this section, we state some estimates of the jump rates of the
trace process on the sets $\ms V^{(p)}$ needed in the next sections.
We assume that the reader is familiar with the notation and results
presented in the appendix.

Fix $1\le p\le \mf q$, and denote by
$\color{blue} \{Y^{n,p}_t: t\ge 0\}$ the trace of
$\{X^{(n)}_t: t\ge 0\} $ on $\ms V^{(p)}$, and by
$\color{blue} R^{(p)}_n : \ms V^{(p)} \times \ms V^{(p)} \to \bb R_+$
its jump rates. By equation (2.5) in \cite{lrev},
\begin{equation}
\label{o-40}
R^{(p)}_n (x,y) \;=\; \lambda_n(x) \; \mb P^n_{\! x} \big[ H_y
= H^+_{\ms V^{(p)}} \big]\;, \quad x\,,\; y \in \ms V^{(p)} \,, 
\; x\not = y  \;.
\end{equation}

Let $r^{(p)}_n(i,j)$, $j\neq i\in S_p$, be the mean rate at which the
trace process $Y^{n,p}_t$ jumps from $\ms V^{(p)}_i$ to
$\ms V^{(p)}_j$:
\begin{equation}
\label{20}
{\color{blue}  r^{(p)}_n(i,j)}
\; := \; \frac{1}{\pi_n(\ms V^{(p)}_i)}
\sum_{x \in\ms V^{(p)}_i} \pi_n(x) 
\sum_{y\in\ms V^{(p)}_j} R^{(p)}_n(x,y) \;.
\end{equation}
By \cite[Theorem 2.7 and 2.12]{lx16}, the sequences
$\theta^{(p)}_n \, r^{(p)}_n(i,j)$ converge for all
$i\not = j\in S_p$.  Denote the limits by $r^{(p)} (i,j)$:
\begin{equation}
\label{o-34}
{\color{blue}  r^{(p)} (i,j)} \;:=\;  \lim_{n\to\infty} \theta^{(p)}_n
\, r^{(p)}_n(i, j) \;\in\; \bb R_+\;.
\end{equation}

Recall from \eqref{54} the definition of the reflection of a Markov
process on a subset of its state space.

\begin{lemma}
\label{l35}
For all $n\ge 1$, $1\le p\le \mf q$, $j\in S_p$, the trace process
$Y^{n,p}_t$ reflected at $\ms V^{(p)}_j$ is irreducible,
\end{lemma}

\begin{proof}
Fix $1\le p\le \mf q$, $j\in S_p$, $x$, $y\in \ms V^{(p)}_j$. We have
to prove that there exists a path $(x=x_0, x_1, \dots, x_\ell=y)$ such
that $x_i\in \ms V^{(p)}_j$, $R^{(p)}_n(x_i,x_{i+1})>0$ for all
$0\le i<\ell$, $n\ge 1$.

By Propositions 6.1 and 6.3 in \cite{bl2}, the trace process
$Y^{n,p}_t$ is an irreducible, $\ms V^{(p)}$-valued continuous-time
Markov chain. Fix $j\in S_p$ and denote by $Y^{n,p,j}_t$ the process
$Y^{n,p}_t$ reflected at $\ms V^{(p)}_j$.

The proof is by induction on $p$. Fix $p=1$ and consider the reflected
process $Y^{n,1,j}_t$ for $j\in S_1$.  By definition of $\ms V^{(1)}$,
the set $\ms V^{(1)}_j$ is a closed irreducible class for the chain
$\bb X_t$. Therefore, for all $x\neq y\in \ms V^{(1)}_j$, there exists
a path $(x=x_0, x_1, \dots, x_\ell=y)$ such that
$x_i\in \ms V^{(1)}_j$, $\bb R_0(x_i,x_{i+1})>0$, $0\le i<\ell$. By
assumptions \eqref{o-01}, \eqref{51}, for all $n\ge 1$,
$R_n(x_i,x_{i+1})>0$ as well, and by \eqref{52},
$R^{(1)}_n(x_i,x_{i+1})>0$, completing the proof for $p=1$.

Fix $p> 1$, and assume that the assertion of the lemma holds for
$1\le q<p$.  Consider the reflected process $Y^{n,p,m}_t$ for
$m\in S_p$, and fix $y\not = x \in \ms V^{(p)}_m$. By definition of
$\ms V^{(p)}_m$, there exists a subset $S_{p,m} \subset S_{p-1}$ such
that $\ms V^{(p)}_m = \cup_{j\in S_{p,m}} \ms V^{(p-1)}_j$.

There are two cases. Assume first that $x$ and $y$ belong to the same
set $\ms V^{(p-1)}_j$. By the induction assumption, there exists a
path $(x=x_0, x_1, \dots, x_\ell=y)$ such that
$x_i\in \ms V^{(p-1)}_j$, $R^{(p-1)}_n(x_i,x_{i+1})>0$ for all
$0\le i<\ell$ and $n\ge 1$. By \eqref{52},
$R^{(p)}_n(x_i,x_{i+1}) \ge R^{(p-1)}_n(x_i,x_{i+1})$, so that
$R^{(p)}_n(x_i,x_{i+1})>0$ for all $0\le i<\ell$ and $n\ge 1$.

Assume now that $x\in \ms V^{(p-1)}_j$ and $y\in \ms V^{(p-1)}_k$ for
$k\neq j \in S_{p,m}$. By construction of $\ms V^{(p)}_m$, there
exists a sequence $(j=j_0, j_1, \dots, j_r=k)$ such that
$j_a\in S_{p,m}$, $r^{(p-1)}(j_a,j_{a+1})>0$, $0\le a < r$.  To keep
the proof simple assume that $r^{(p-1)}(j,k)>0$. The reader will see
that the proof in the general case is similar.

Since $r^{(p-1)}(j,k)>0$, by \eqref{o-34}, \eqref{20} and
\eqref{o-58}, there exists $x'\in \ms V^{(p-1)}_j$,
$y'\in \ms V^{(p-1)}_k$ such that $R^{(p-1)}_n(x',y')>0$ for $n$
sufficiently large. By \eqref{51} and \eqref{53},
$R^{(p-1)}_n(x',y')>0$ for all $n\ge 1$. Hence, by \eqref{52},
$R^{(p)}_n(x',y')>0$ for all $n\ge 1$. We may now repeat the argument
presented in the previous paragraph to construct a path in the set
$\ms V^{(p-1)}_j$ from $x$ to $x'$, and a second one in the set
$\ms V^{(p-1)}_k$ from $y'$ to $y$. Chaining the paths yields a path
$(x=z_0, z_1, \dots, z_\ell=y)$ such that $z_i\in \ms V^{(p)}_m$,
$R^{(p)}_n(z_i,z_{i+1})>0$ for all $0\le i<\ell$ and $n\ge 1$. This
completes the proof of the lemma.
\end{proof}

\begin{lemma}
\label{l18}
Fix $1\le p \le \mf q$. Then, for all $x$, $y\in \ms V^{(p)}$,
\begin{equation*}
\lim_{n\to\infty}  R^{(p)}_n(x,y) \;=\; \bb R_0(x,y)  \;.
\end{equation*}
\end{lemma}

\begin{proof}
Fix $1\le p \le \mf q$ and $x$, $y\in \ms V^{(p)}$.  By \eqref{o-40},
decomposing the probability appearing on the right-hand side of this
equation according to the first jump yields that
\begin{equation*}
R^{(p)}_n (x,y) \;=\;
R_n(x,y) \;+\;
\sum_{z\neq y} R_n(x,z) \, \mb P^n_{\! z} \big[ H_y
= H_{\ms V^{(p)}} \big] \;.
\end{equation*}
The first term converges to $\bb R_0(x,y)$. As $x\in \ms V^{(p)}$ and
(by the tree construction) $\ms V^{(p)}$ is the union of some sets
$\ms V_k$, $k\in S_1$, $x\in \ms V_j$ for some $j\in S_1$.  The
probability on the second term vanishes if $z\in \ms V^{(p)}$. We may
therefore restrict the sum to $z\not \in \ms V^{(p)}$, or to
$\ms V^c_j$ (because $\ms V_j \subset \ms V^{(p)})$. However, by
definition of $\ms V_j $, $R_n(x,z) \to 0$ for all
$z\not \in \ms V_j $. Thus, the second term of the previous displayed
formula vanishes, which completes the proof of the lemma.
\end{proof}

\begin{lemma}
\label{l17}
Fix $1\le p\le \mf q$. Then,
\begin{equation*}
\limsup_{n\to \infty} \theta^{(p)}_n\, R^{(p)}_n(x,y) \;<\; \infty
\end{equation*}
for all $k\not =j \in S_p$, $x\in\ms V^{(p)}_j$, $y\in \ms V^{(p)}_k$.
\end{lemma}

\begin{proof}
As $y\in \ms V^{(p)}_k$,
$R^{(p)}_n(x,y) \le R^{(p)}_n(x,\ms V^{(p)}_k)$.  By \eqref{o-58},
there exists a finite constant $C_0$ such that
\begin{equation*}
\theta^{(p)}_n\, R^{(p)}_n(x,y) \;\le\; C_0\, \theta^{(p)}_n\, 
\sum_{z\in \ms V^{(p)}_j} \frac{\pi_n(z)}{\pi_n(\ms V^{(p)}_{j})}
R^{(p)}_n(z,\ms V^{(p)}_k)\;.
\end{equation*}
It remains to recall \eqref{o-34} to complete the proof.
\end{proof}

\begin{lemma}
\label{l34}
For all $1\le p\le \mf q$, $k\not = j \in S_1$ such that
$\ms V^{(1)}_j \cup \ms V^{(1)}_k \subset \ms V^{(p)}$,
$x\in \ms V^{(1)}_j$, the sequence
$\theta^{(1)}_n\, R^{(p)}_n(x,\ms V^{(1)}_k)$ is bounded.
\end{lemma}

\begin{proof}
The proof is by induction on $p$. For $p=1$, the assertion of the
lemma follows from \eqref{o-34}, \eqref{20} and \eqref{o-58}.

Fix $p> 1$, and assume that the assertion of the lemma holds for
$1\le q<p$. Fix $k\not = j \in S_1$ such that
$\ms V^{(1)}_j \cup \ms V^{(1)}_k \subset \ms V^{(p)}$,
$x\in \ms V^{(1)}_j$. By \eqref{o-40},
\begin{equation*}
R^{(p)}_n (x,\ms V^{(1)}_k) \;=\;
\lambda_n(x) \; \mb P^n_{\! x} \big[ H_{\ms V^{(1)}_k}
= H^+_{\ms V^{(p)}} \big]\;.
\end{equation*}

Recall that $\ms V^{(p)} \subset\ms V^{(p-1)}$.  Assume that
$H^+_{\ms V^{(p)}} = H^+_{\ms V^{(p-1)}}$. Later we consider the case
$H^+_{\ms V^{(p)}} > H^+_{\ms V^{(p-1)}}$. In the first case, we need
to estimate
\begin{equation*}
\begin{aligned}
& \lambda_n(x) \; \mb P^n_{\! x} \big[ H_{\ms V^{(1)}_k}
= H^+_{\ms V^{(p)}} \,,\, H^+_{\ms V^{(p)}} = H^+_{\ms V^{(p-1)}}
\big] \\
&\quad
\le\; \lambda_n(x) \; \mb P^n_{\! x} \big[ H_{\ms V^{(1)}_k}
= H^+_{\ms V^{(p-1)}}  \big] \;=\;
R^{(p-1)}_n (x,\ms V^{(1)}_k) \;.
\end{aligned}
\end{equation*}
By the induction hypothesis this later quantity multiplied by 
$\theta^{(1)}_n$ is bounded.

It remains to estimate the expression
\begin{equation*}
\lambda_n(x) \; \mb P^n_{\! x} \big[ H^+_{\ms V^{(p-1)}} < H^+_{\ms V^{(p)}} 
\big]  \;.
\end{equation*}
By construction, there exists $S'_{p-1} \subset S_{p-1}$ such that
$\ms V^{(p-1)} \setminus \ms V^{(p)} = \cup_{\ell\in S'_{p-1}} \ms
V^{(p-1)}_\ell$. Mind that $S'_{p-1}$ consists of the transient points
of the $S_{p-1}$-valued Markov chain $\bb X^{(p-1)}_t$.  Since
$x\in \ms V^{(p)} \subset \ms V^{(p-1)}$, let
$m\in S_{p-1} \setminus S'_{p-1}$ such that $x\in \ms V^{(p-1)}_{m}$.
With this notation, the previous term is bounded by
\begin{equation*}
\sum_{\ell\in S'_{p-1}}
\lambda_n(x) \; \mb P^n_{\! x}
\big[ H_{\ms V^{(p-1)}_\ell}  = H^+_{\ms V^{(p-1)}} \big]
\;=\; \sum_{\ell\in S'_{p-1}} R^{(p-1)}_n (x,\ms V^{(p-1)}_\ell) \;.
\end{equation*}
By\eqref{o-34}, \eqref{20} and \eqref{o-58}, the limit as $n\to\infty$
of the previous expression multiplied by $\theta^{(p-1)}_n$ is bounded by
\begin{equation*}
\sum_{\ell\in S'_{p-1}} r^{(p-1)} (m,\ell) \;.
\end{equation*}
This sum vanishes because $m$ is a recurrent point of the chains
$\bb X^{(p-1)}_t$ and $S'_{p-1}$ is a transient subset. To complete
the proof of the lemma, it remains to recall that
$\theta^{(1)}_n \le \theta^{(p-1)}_n$.
\end{proof}

\begin{lemma}
\label{l20}
Fix $1\le p < q\le \mf q$.  Then, $r^{(p)}(i,k) \,=\, 0$ for all
$k \neq i \in S_p$ such that $\ms V^{(p)}_k \not\subset \ms V^{(q)}$,
$\ms V^{(p)}_i\subset \ms V^{(q)}$.
\end{lemma}

\begin{proof}
As $\ms V^{(p)}_k \subset \ms V^{(p)}$ and
$\ms V^{(p)}_k \not\subset \ms V^{(q)}$, by the tree construction
there exists $p\le p'<q$ such that
$\ms V^{(p)}_k \subset \ms V^{(p')}$ and
$\ms V^{(p)}_k \not\subset \ms V^{(p'+1)}$. Since
$\ms V^{(p)}_k \subset \ms V^{(p')}$, there exists $\ell\in S_{p'}$
such that $\ms V^{(p)}_k \subset \ms V^{(p')}_\ell$. As
$\ms V^{(p)}_k \not\subset \ms V^{(p'+1)}$, $\ell$ is a transient
state for the Markov chain $\bb X^{(p')}_t$.

On the other hand, as $\ms V^{(p)}_i \subset \ms V^{(q)}$ and
$\ms V^{(q)} \subset \ms V^{(p')}$,
$\ms V^{(p)}_i \subset \ms V^{(p')}$. Thus, there exists
$m \in S_{p'}$ such that $\ms V^{(p)}_i \subset \ms V^{(p')}_m$. As
$\ms V^{(p)}_i \subset \ms V^{(q)} \subset \ms V^{(p'+1)}$, $m$ is a
recurrent state for the Markov chain $\bb X^{(p')}_t$. In particular,
$m\not = \ell$.

As $m$ is recurrent and $\ell$ transient for the Markov chain $\bb
X^{(p')}_t$, $r^{(p')}(m,\ell)=0$. Thus, by \eqref{o-34} and
\eqref{20},
\begin{equation*}
0\;=\; \lim_{n\to\infty} \theta^{(p')}_n
\, r^{(p')}_n(m, \ell)
\;=\; \lim_{n\to\infty} \theta^{(p')}_n \,
\frac{1}{\pi_n(\ms V^{(p')}_m)}
\sum_{x \in\ms V^{(p')}_m} \pi_n(x) 
\, R^{(p')}_n(x,\ms V^{(p')}_\ell)\;.
\end{equation*}
By \eqref{o-40}, for $n$ fixed the expression on right-hand side is
equal to
\begin{equation}
\label{21}
\theta^{(p')}_n \, \frac{1}{\pi_n(\ms V^{(p')}_m)}
\sum_{x \in\ms V^{(p')}_m} \pi_n(x) \, \lambda_n(x)
\, \mb P^n_{\! x} \big[ H_{\ms V^{(p')}_\ell}
= H^+_{\ms V^{(p')}} \big] \;.
\end{equation}

Since $\ms V^{(p)}_k \subset \ms V^{(p')}_\ell$ and
$\ms V^{(p)} \supset \ms V^{(p')}$,
\begin{equation*}
\mb P^n_{\! x} \big[ H_{\ms V^{(p)}_k} = H^+_{\ms V^{(p)}} \big]
\;\le \;
 \mb P^n_{\! x} \big[ H_{\ms V^{(p')}_\ell}
= H^+_{\ms V^{(p')}} \big] \;. 
\end{equation*}
Hence, as $\theta^{(p')}_n \ge \theta^{(p)}_n$ and $\ms V^{(p')}_m
\supset \ms V^{(p)}_i$, by \eqref{o-58}, \eqref{21} is bounded below by
\begin{equation*}
c_0\, \theta^{(p)}_n \, 
\sum_{x \in\ms V^{(p)}_i}  \lambda_n(x)
\, \mb P^n_{\! x} \big[ H_{\ms V^{(p)}_k}
= H^+_{\ms V^{(p)}} \big] 
\end{equation*}
for some positive constant $c_0$.
This expression is clearly bounded below by
\begin{equation*}
c_0\, \theta^{(p)}_n \, \frac{1}{\pi_n(\ms V^{(p)}_i)}
\sum_{x \in\ms V^{(p)}_i} \pi_n(x) \, \lambda_n(x)
\, \mb P^n_{\! x} \big[ H_{\ms V^{(p)}_k}
= H^+_{\ms V^{(p)}} \big] \;=\;
c_0\, r^{(p)}_n (i,k) \;.
\end{equation*}
Collecting the previous estimates yields that this expression vanishes
as $n\to\infty$, as claimed. 
\end{proof}

For $1\le p < q\le \mf q$ and $i \not = j\in S_p$. Assume that
$\ms V^{(p)}_i$ and $\ms V^{(p)}_j$ are contained in $\ms V^{(q)}$:
$\ms V^{(p)}_i \cup \ms V^{(p)}_j \subset \ms V^{(q)}$. Let
\begin{equation*}
{\color{blue}  r^{p,q}_n(i,j)}
\; := \; \frac{1}{\pi_n(\ms V^{(p)}_i)}
\sum_{x \in\ms V^{(p)}_i} \pi_n(x) 
\sum_{y\in\ms V^{(p)}_j} R^{(q)}_n(x,y) \;.
\end{equation*}
The difference with respect to $r^{p}_n(i,j)$ is that we replaced
$R^{(p)}_n(x,y)$ by $R^{(q)}_n(x,y)$, that is, the trace on
$\ms V^{(p)}$ by the one on the smaller set $\ms V^{(q)}$.

\begin{corollary}
\label{l19}
Fix $1\le p < q\le \mf q$, $i \not = j\in S_p$. Assume that there
exists $m\in S_q$ such that
$\ms V^{(p)}_i \cup \ms V^{(p)}_j \subset \ms V^{(q)}_m$. Then, 
\begin{equation*}
\lim_{n\to\infty} \theta^{(p)}_n
\, r^{p,q}_n(i, j) \;=\; r^{(p)} (i,j) \;.
\end{equation*}
\end{corollary}

\begin{proof}
By \eqref{o-34}, \eqref{20} and \eqref{o-40},
\begin{equation}
\label{21b}
r^{p}(i, j) \;=\; \lim_{n\to\infty}
\frac{\theta^{(p)}_n}{\pi_n(\ms V^{(p)}_i)}
\sum_{x \in\ms V^{(p)}_i} \pi_n(x) 
\lambda_n(x) \,
\mb P^n_{\! x} \big[ H_{\ms V^{(p)}_j}
= H^+_{\ms V^{(p)}} \big]\;. 
\end{equation}

Let $S_{p,q} := \{k\in S_p : \ms V^{(p)}_k \subset \ms V^{(q)}\}$,
$\ms V^{(p,q)} := \cup_{k\in S_{p,q}}\ms V^{(p)}_k$,
$\ms U^{(p,q)} := \ms V^{(p)} \setminus \ms V^{(q)}$. By Lemma
\ref{l20},
\begin{equation}
\label{22}
\lim_{n\to\infty}
\frac{\theta^{(p)}_n}{\pi_n(\ms V^{(p)}_i)}
\sum_{x \in\ms V^{(p)}_i} \pi_n(x) 
\lambda_n(x) \,
\mb P^n_{\! x} \big[ H_{\ms U^{(p,q)}}
= H^+_{\ms V^{(p)}} \big] \;=\; 0 \;. 
\end{equation}
Since $\ms V^{(p)} = \ms V^{(p,q)} \cup \ms U^{(p,q)}$,
$\ms V^{(p,q)} \cap \ms U^{(p,q)} = \varnothing$, the sets
$\{ H_{\ms U^{(p,q)}} = H^+_{\ms V^{(p)}} \}$ and
$\{ H^+_{\ms V^{(q)}} = H^+_{\ms V^{(p)}} \}$ form a partition of the
space. Decomposing the event appearing in \eqref{21b} according to this
partition, by \eqref{22},
\begin{equation*}
r^{p}(i, j) \;=\; \lim_{n\to\infty}
\frac{\theta^{(p)}_n}{\pi_n(\ms V^{(p)}_i)}
\sum_{x \in\ms V^{(p)}_i} \pi_n(x) 
\lambda_n(x) \,
\mb P^n_{\! x} \big[ H_{\ms V^{(p)}_j}
= H^+_{\ms V^{(q)}} \,,\, H^+_{\ms V^{(q)}}= H^+_{\ms V^{(p)}}
\big]\;. 
\end{equation*}
By \eqref{22} once more,
\begin{equation*}
r^{p}(i, j) \;=\; \lim_{n\to\infty}
\frac{\theta^{(p)}_n}{\pi_n(\ms V^{(p)}_i)}
\sum_{x \in\ms V^{(p)}_i} \pi_n(x) 
\lambda_n(x) \,
\mb P^n_{\! x} \big[ H_{\ms V^{(p)}_j}
= H^+_{\ms V^{(q)}} \big]  \;. 
\end{equation*}
By definition, the right-hand side is $\lim_{n\to\infty}
\theta^{(p)}_n \,r^{p,q}_n(i, j)$, which completes the proof of the
lemma. 
\end{proof}

\begin{lemma}
\label{l21}
Fix $1\le p < q\le \mf q$.  Then,
\begin{equation*}
r^{(p)}(i,k) \;=\;  0 
\end{equation*}
for all $k \neq i \in S_p$ such that
$\ms V^{(p)}_i \subset \ms V^{(q)}_a$,
$\ms V^{(p)}_k\subset \ms V^{(q)}_b$ for some $a\neq b\in S_q$.
\end{lemma}

\begin{proof}
Fix $1\le p < q\le \mf q$, and $k \neq i \in S_p$ such that
$\ms V^{(p)}_i \subset \ms V^{(q)}_a$,
$\ms V^{(p)}_k\subset \ms V^{(q)}_b$ for some $a\neq b\in S_q$.
Both states $i$ and $k$ are recurrent for the chain $\bb X^{(p)}_t$
because if one of them was transient it would not belong to $\ms
V^{(p+1)}\supset \ms V^{(q)}$.

Suppose by contradiction that $r^{(p)}(i,k)>0$. Hence, since both
states are recurrent, they belong to the same irreducible class. In
particular, there exists $m\in S_{p+1}$ such that
$\ms V^{(p)}_i \cup \ms V^{(p)}_k\subset \ms V^{(p+1)}_m$, so that
$\ms V^{(p)}_i \cup \ms V^{(p)}_k\subset \ms V^{(q)}_c$ for some
$c \in S_q$, in contradictions with the hypotheses.
\end{proof}

\section{$\Gamma-\limsup$ of the trace}
\label{sec3}

The main result of this section, Proposition \ref{l16}, states that
$\ms I^{(p)}$ is a $\Gamma-\limsup$ for the sequence
$\theta^{(p)}_n \ms I^{(p)}_n$. Here, $\ms I^{(p)}_n$ stands for the
large deviations rate functionals of the trace processes $Y^{n,p}_t$.
We assume below that the reader is familiar with the notation and
results presented in the appendix.

Fix $1\le p\le \mf q$.  Denote by
$\ms I^{(p)}_n \colon \ms P(\ms V^{(p)}) \to [0,+\infty)$ the
occupation time large deviations rate functional of the trace process
$Y^{n,p}_t$:
\begin{equation}
\label{eq:2}
{\color{blue}  \ms I^{(p)}_n (\mu)} \;:=\;
\sup_{H} \,-\,  \sum_{x\in \ms V^{(p)}} \mu(x)\,
e^{- H(x) } \, [\, (\mf T_{\ms V^{(p)}} \ms L_n) \, e^{H})\,]
(x)  \;,
\end{equation}
where the supremum is carried over all functions
$H\colon \ms V^{(p)}\to \bb R$ and $\mf T_{\ms V^{(p)}} \ms L_n$,
introduced and examined in Appendix \ref{sec-a3}, is the generator of
the trace process $Y^{n,p}_t$. The main result of this section reads
as follows.

\begin{proposition}
\label{l16}
For all $\mu\in \ms P(\ms V^{(p)})$, there exists a sequence of
measures $\mu_n \in \ms P(\ms V^{(p)})$ such that $\mu_n\to \mu$ and
\begin{equation*}
\limsup_{n\to\infty} \theta^{(p)}_n\, \ms I^{(p)}_n(\mu_n) \;\le\; \ms
I^{(p)}(\mu)\;.  
\end{equation*}
\end{proposition}

The proof of Proposition \ref{l16} is divided in several lemmata.  Fix
$1\le p\le \mf q$ and a measure $\mu \in \ms P(\ms V^{(p)})$ which can
be represented as $\mu = \sum_{j\in S_p} \omega_j \, \pi^{(p)}_j$ for
some $\omega\in \ms P(S_p)$ such that $\omega_j >0$ for all
$j\in S_p$.  We first construct a sequence
$\mu_n \in \ms P(\ms V^{(p)})$ which converges to $\mu$.

Denote by $\color{blue} \mf Q^{(p)}_a$, $1\le a \le \ell_p$, the
equivalent classes of the Markov chain $\bb X^{(p)}_t$, by
$\color{blue} \mf D^{(p)}_a$, $1\le a \le m_{p}$, the ones which are
not singletons, and by $\color{blue} S^{\rm sgl}_p$ the set of states
$j\in S_p$ such that $\{j\}$ is an equivalent class. Clearly,
\begin{equation}
\label{27}
S_p \;=\;
\bigcup_{a=1}^{\ell_p} \mf Q^{(p)}_a
\;=\; S^{\rm sgl}_p \;\cup\; \bigcup_{a=1}^{m_p} \mf D^{(p)}_a\;. 
\end{equation}

Recall from Appendix \ref{sec01} the definition of a Markov chain
reflected at a set and the notation used to represent its generator.
Denote by $\bb L^{(p)}_a$, $1\le a\le m_p$, the generator
$\bb L^{(p)}$ reflected at $\mf D^{(p)}_a$:
$\color{blue} \bb L^{(p)}_a := \mf R_{\mf D^{(p)}_a}\, \bb L^{(p)}$.
As $\omega_j>0$ for all $j\in \mf D^{(p)}_a$, by Lemma \ref{l02},
there exists $\color{blue} \bs h_a \colon \mf D^{(p)}_a \to \bb R$
which solves the optimal problem \eqref{19} for
$I_{\bb L^{(p)}_a} (\omega) $.

Let
$\color{blue} \ms W^{(p)}_a := \cup_{j\in \mf D^{(p)}_a} \ms
V^{(p)}_j$, $1\le \ell\le m_p$. The generator of the trace process
$Y^{n,p}_t$ reflected at the set $\ms W^{(p)}_a$ is denoted by
$\mf R_{\ms W^{(p)}_a} \mf T_{\ms V^{(p)}}\, \ms L_n$.

\begin{lemma}
\label{l36}
The Markov chain associated to the generator
$\mf R_{\ms W^{(p)}_a} \mf T_{\ms V^{(p)}}\, \ms L_n$ is irreducible.
\end{lemma}

\begin{proof}
Recall from \cite[Proposition 6.1]{bl2} that the
Markov chain induced by the trace generator
$\mf T_{\ms V^{(p)}}\, \ms L_n$ is irreducible.  Since $\mf D^{(p)}_a$
is an equivalent class for the Markov chain $\bb X^{(p)}_t$, the
argument presented in the proof of Lemma \ref{l35} yields that the
Markov chain induced by the reflected generator
$\mf R_{\ms W^{(p)}_a} \, \mf T_{\ms V^{(p)}}\, \ms L_n$ is also
irreducible. 
\end{proof}

Denote by $H_a\colon \ms W^{(p)}_a \to \bb R$ the function given by
\begin{equation}
\label{33}
{\color{blue} H_a}
\;=\; \sum_{j\in \mf D^{(p)}_a} \bs h_a(j)\, \chi_{\ms V^{(p)}_j} \;,
\end{equation}
where, recall, $\chi_{\ms A}$ stands for the indicator function of the
set $\ms A$. Recall from \eqref{56} the definition of a tilted
generator $\mf M_G\, L$, and consider the generator
$\mf M_{H_a}\, \mf R_{\ms W^{(p)}_a} \mf T_{\ms V^{(p)}}\, \ms L_n$.
Since tilting the generator does not affect its irreducibility, it
follows from the previous result that the Markov chain associated to
the generator
$\mf M_{H_a}\, \mf R_{\ms W^{(p)}_a} \mf T_{\ms V^{(p)}}\, \ms L_n$ is
also irreducible.  Denote by
$\color{blue} \mu^{p,a}_n \in \ms P(\ms W^{(p)}_a)$ its stationary
state.

\begin{lemma}
\label{l12}
The sequence of probability measures $\mu^{p,a}_n$ converges to
$\sum_{j\in \mf D^{(p)}_a} \omega^{(a)}_j \, \pi^{(p)}_j$, where
$\color{blue} \omega^{(a)}_j = \omega_j/ \Omega_a$,
$\color{blue} \Omega_a = \sum_{j\in \mf D^{(p)}_a} \omega_j$.
\end{lemma}

\begin{proof}
Since $\ms P(\ms W^{(p)}_a)$ is compact for the weak topology, it is
enough to prove uniqueness of limit points.  consider a subsequence of
$\mu^{p,a}_n$, still denoted by $\mu^{p,a}_n$, which converges to a
limit denoted by $\nu\in \ms P(\ms W^{(p)}_a)$.

\smallskip\noindent {\bf Step 1: $\nu$ on the sets $\ms V^{(1)}_i$.}
The set $\ms W^{(p)}_a$ is the union of sets $\ms V^{(p)}_i$, which in
turn are formed by sets $\ms V^{(1)}_j$. Hence,
$\ms W^{(p)}_a = \cup_{i\in S^a_1} \ms V^{(1)}_i$ for some subset
$\color{blue} S^a_1$ of $S_1$.

Fix $x\in \ms W^{(p)}_a$ and assume that $x\in \ms V^{(1)}_k$ for some
$k\in S^a_1$. Since $\mu^{p,a}_n$ is the stationary state for the
chain induced by the generator
$\mf M_{H_a}\, \mf R_{\ms W^{(p)}_a} \, \mf T_{\ms V^{(p)}}\,\ms L_n$,
\begin{equation*}
\sum_{y\in \ms W^{(p)}_a} \mu^{p,a}_n(y)\, R^{(p)}_n(y,x) \, e^{H_a(x)-H_a(y)} \;=\;
\sum_{y\in \ms W^{(p)}_a} \mu^{p,a}_n(x)\, R^{(p)}_n(x,y) \, e^{H_a(y)-H_a(x)} \;.
\end{equation*}
Since $\mu^{p,a}_n \to \nu$,
taking $n\to\infty$ in the previous formula, by Lemma \ref{l18},
\begin{equation*}
\sum_{y\in \ms W^{(p)}_a} \nu (y)\, \bb R_0 (y,x) \, e^{H_a(x)-H_a(y)} \;=\;
\sum_{y\in \ms W^{(p)}_a} \nu (x)\, \bb R_0 (x,y) \, e^{H_a(y)-H_a(x)} \;.
\end{equation*}
By definition of the sets $\ms V^{(1)}_i$, $\bb R_0(z,w)=0$ if
$z\in \ms V^{(1)}_i$, $w\in \ms V^{(1)}_{i'}$ and $i\neq i'$. Hence,
as $x\in \ms V^{(1)}_{k}$ and $H_a$ is contant on each set
$\ms V^{(p)}_{j}$, the previous identity becomes
\begin{equation*}
\sum_{y\in \ms V^{(1)}_{k}} \nu (y)\, \bb R_0 (y,x)  \;=\;
\sum_{y\in \ms V^{(1)}_{k}} \nu(x)\, \bb R_0 (x,y) \;.
\end{equation*}
Therefore, the measure $\nu$ restricted to $\ms V^{(1)}_{k}$ is a
stationary measure for the Markov chain $\bb X_t$ restricted to
$\ms V^{(1)}_{k}$. Since this process is irreducible on
$\ms V^{(1)}_{k}$, by the definition of $\pi^{(1)}_k$ given right
after \eqref{o-52},
$\nu (\,\cdot\,) \,=\, \nu (\ms V^{(1)}_{k}) \, \pi^{(1)}_k
(\,\cdot\,)$. Hence, $\nu$ is a convex combination of the stationary
states $\pi^{(1)}_\ell$:
\begin{equation}
\label{23}
\nu \;=\; \sum_{k\in S^a_1} \vartheta_1 (k)\, \pi^{(1)}_k 
\end{equation}
for some probability measure $\vartheta_1$ on $S^a_1$.

\smallskip\noindent {\bf Step 2: An equation for $\vartheta_1$.}  Fix
$1\le r\le p$, and let
$\color{blue} S^a_r := \{ i\in S_r : \ms V^{(r)}_i \subset \ms
W^{(p)}_a\}$ so that
$\ms W^{(p)}_a = \cup_{i\in S^a_r} \ms V^{(r)}_i$.  Fix
$\bs g\colon S^a_r \to \bb R$ and let
$G\colon \ms W^{(p)}_a \to \bb R$ be given by
$G = \sum_{j\in S^a_r} \bs g(j)\, \chi_{\ms V^{(r)}_j}$. As
$\mu^{p,a}_n$ is a stationary state,
\begin{equation*}
\sum_{x\in\ms W^{(p)}_a} \mu^{p,a}_n(x)
\sum_{y\in \ms W^{(p)}_a} R^{(p)}_n(x,y) \, e^{H_a(y) - H_a(x)}\,
\big[ \, G(y)\,-\, G(x)\,\big]\;=\; 0\;.
\end{equation*}
Since $H_a$ is constant on the sets $\ms V^{(p)}_\ell$, it is also
constant on the sets $\ms V^{(r)}_j$, which are subsets of the former
sets.  By definition of $G$ and $H_a$, this identity can be written as
\begin{equation}
\label{17b}
\sum_{j\in S^a_r}\sum_{k\in S^a_r\setminus\{j\}}
e^{\bs f_a(k) - \bs f_a(j)}\,
\big[ \, \bs g(k) \,-\, \bs g (j) \,\big]
\sum_{x\in\ms V^{(r)}_j} \mu^{p,a}_n(x) \, R^{(p)}_n(x,\ms V^{(r)}_k)
\;=\; 0\;.
\end{equation}
Here $\bs f_a \colon S^a_r \to \bb R$ is the function defined by
$\bs f_a(j) = \bs h_a(m)$ for all $j\in S^{a,m}_r$, where
$\color{blue} S^{a,m}_r := \{ i\in S_r : \ms V^{(r)}_i \subset \ms
V^{(p)}_m\}$, $m\in S^a_p$.

If $p=1$, jump to Step 5.  Assume below that $p>1$ and set $r=1$. We
claim that for each $j\neq k \in S^a_1$,
\begin{equation}
\label{16b}
\lim_{n\to\infty} \theta^{(1)}_n\,
\sum_{x\in\ms V^{(1)}_j} \mu^{p,a}_n(x) \, R^{(p)}_n(x,\ms V^{(1)}_k) \;=\;
\vartheta_1 (j)\, r^{(1)}(j,k)\;,
\end{equation}
where $\vartheta_1 \in \ms P(S^a_1)$ is the probability measure obtained in the
first step. 

To prove \eqref{16b}, rewrite the expression on the left-hand side as
\begin{equation*}
\theta^{(1)}_n\,
\sum_{x\in\ms V^{(1)}_j} \frac{\pi_n(\ms V^{(1)}_j)} {\pi_n(x)} \,
\mu^{p,a}_n(x) \, \frac{\pi_n(x)}{\pi_n(\ms V^{(1)}_j)}
\, R^{(p)}_n(x,\ms V^{(1)}_k) \;.
\end{equation*}
By the first part of the proof,
$\mu^{p,a}_n(x) \to \vartheta_1 (j) \, \pi^{(1)}_j(x)$. By
\eqref{o-58}, $\pi_n(\ms V^{(1)}_j)/\pi_n(x) \to 1/
\pi^{(1)}_j(x)$. By Lemma \ref{l34}, the sequence
$\theta^{(1)}_n\, R^{(p)}_n(x,\ms V^{(1)}_k)$ is bounded. Therefore,
by Corollary \ref{l19}, \eqref{16b} holds.

By \eqref{17b} and \eqref{16b},
\begin{equation*}
\sum_{j\in S^a_1}\sum_{k\in S^a_1\setminus\{j\}}
\vartheta_1 (j) \, r^{(1)}(j,k) \,
e^{\bs f_a(k) - \bs f_a(j)}\,
\big[ \, \bs g(k) \,-\, \bs g (j) \,\big]
\;=\; 0
\end{equation*}
for all functions $\bs g\colon S^a_1 \to \bb R$.

By Lemma \ref{l21}, $r^{(1)}(j,k)=0$ for all $j\in S^{a,m}_1$,
$k\in S^{a,\ell}_1$ and $m\not = \ell$. We may therefore rewrite the
previous identity as
\begin{equation*}
\sum_{m\in S^a_p} \sum_{j\in S^{a,m}_1}\sum_{k\in S^{a,m}_1\setminus\{j\}}
\vartheta_1 (j) \, r^{(1)}(j,k) \,
e^{\bs f_a(k) - \bs f_a(j)}\,
\big[ \, \bs g(k) \,-\, \bs g (j) \,\big]
\;=\; 0\;.
\end{equation*}
As the function $\bs f_a$ is constant on each $S^{a,m}_1$, we may
remove the exponential from the previous equation and rewrite the
identity as
\begin{equation*}
\sum_{j\in S^a_1}\sum_{k\in S^a_1\setminus\{j\}}
\vartheta_1 (j) \, r^{(1)}(j,k) \,
\big[ \, \bs g(k) \,-\, \bs g (j) \,\big]
\;=\; 0
\end{equation*}
for all functions $\bs g\colon S^a_1 \to \bb R$.  Hence,
$\vartheta_1 (\cdot)$ is a stationary state for the chain
$\bb X^{(1)}_t$ reflected at $S^a_{1}$.

\smallskip\noindent{\bf Step 3: from $\vartheta_1$ to $\vartheta_2$}
Recall that if $p=1$, the proof continues at Step 5. By its
definition, given a few lines above equation \eqref{o-80},
$\color{blue} M^{(1)}_m$, $m\in S_2$, represents the stationary state
of the chain $\bb X^{(1)}_t$ reflected at
$S^{2,m}_1 = \{j\in S_1: \ms V^{(1)}_j \subset \ms V^{(2)}_m\}$.
Since $\vartheta_1$ is a stationary state of the chain $\bb X^{(1)}_t$
whose support is contained in $S^a_1$, $\vartheta_1$ is a convex
combination of the measures $M^{(1)}_m$, $m \in S^a_2$:
\begin{equation*}
\vartheta_1 (\,\cdot\,) \;=\;  \sum_{m \in S^a_2} \vartheta_2(m)\,
M^{(1)}_m(\,\cdot\,)
\end{equation*}
for a probability measure $\vartheta_2 \in \ms P(S^a_2)$. Thus,
\begin{equation*}
\nu(\,\cdot\,) \;=\; \sum_{k\in S^a_1} \vartheta_1 (k)\, \pi^{(1)}_k (\,\cdot\,)
\;=\; \sum_{k\in S^a_1}
\sum_{m \in S^a_2} \vartheta_2(m)\, M^{(1)}_m(k) \, \pi^{(1)}_k
(\,\cdot\,)\;. 
\end{equation*}
Changing the order of summation, and recalling the definition of the
measures $\pi^{(2)}_m$ yields that
\begin{equation}
\label{24}
\nu(\,\cdot\,) \;=\; \sum_{m \in S^a_2}
\vartheta_2(m)\,  \sum_{k\in S^{2,m}_1}
M^{(1)}_m(k) \, \pi^{(1)}_k (\,\cdot\,)
\;=\; \sum_{m \in S^a_2}
\vartheta_2(m)\,   \pi^{(2)}_m (\,\cdot\,) \;. 
\end{equation}

Step 2 and 3 permitted to pass from \eqref{23} to \eqref{24}. That is,
at the end of Step 1 we obtained that $\nu$ is a convex combination of
the measures $\pi^{(1)}_j$. We now have shown that it is, actually, a
convex combination of the measures $\pi^{(2)}_j$. Next step consist in
iterating this argument to obtain that $\nu$ is, actually, a convex
combination of the measures $\pi^{(p)}_j$.

\smallskip\noindent{\bf Step 4: An iteration.}  If $p=2$ the proof
continues at Step 5.  Assume here that $p>2$, and suppose that we
proved that
\begin{equation*}
\nu(\,\cdot\,) \;=\; \sum_{j \in S^a_s}
\vartheta_s(j)\,   \pi^{(s)}_j (\,\cdot\,)  
\end{equation*}
for some $2\le s<p$ and some probability measure
$\vartheta_s \in \ms P(S^a_s)$.

Resume the proof of Step 2 at equation \eqref{17b} with $r=s$.  Recall
the argument presented to derive \eqref{16b}.  Fix
$j\not = k\in S^a_s$. Inserting $\pi_n(\ms V^{(s)}_j)/\pi_n(x)$
instead of $\pi_n(\ms V^{(1)}_j)/\pi_n(x)$ yields that
\begin{equation}
\label{31}
\lim_{n\to\infty} \theta^{(s)}_n\,
\sum_{x\in\ms V^{(s)}_j} \mu^{p,a}_n(x) \, R^{(p)}_n(x,\ms V^{(s)}_k) \;=\;
\vartheta_s (j)\, r^{(s)}(j,k)\;.
\end{equation}
By Lemma \ref{l21}, $r^{(s)}(j,k) = 0$ if $j\in S^{a,\ell}_s$,
$k\in S^{a,m}_s$ and $\ell \neq m$. Hence, applying the arguments
presented in Step 3, one gets that
\begin{equation*}
\nu(\,\cdot\,) \;=\; \sum_{j \in S^a_{s+1}}
\vartheta_{s+1}(j)\,   \pi^{(s+1)}_j (\,\cdot\,) 
\end{equation*}
for some some probability measure $\vartheta_{s+1} \in \ms P(S^a_{s+1})$.

Iterating this procedure yields that
\begin{equation}
\label{25}
\nu(\,\cdot\,) \;=\; \sum_{j \in S^a_{p}}
\vartheta_{p}(j)\,   \pi^{(p)}_j (\,\cdot\,) 
\end{equation}
for some some probability measure $\vartheta_{p} \in \ms P(S^a_{p})$.

\smallskip\noindent{\bf Step 5: Conclusion.}
Applying again the arguments presented in Step 2 yields that
\begin{equation*}
\sum_{j\in S^a_p}\sum_{k\in S^a_p\setminus\{j\}}
\vartheta_{p} (j) \, r^{(p)}(j,k) \,
e^{\bs h_a(k) - \bs h_a(j)}\,
\big[ \, \bs g(k) \,-\, \bs g (j) \,\big]
\;=\; 0
\end{equation*}
for all function $\bs g\colon S^a_p \to \bb R$. In particular,
$\vartheta_{p}$ is a stationary state for the $S^a_p$-valued Markov chain
which jumps from $j$ to $k$ at rate
$r^{(p)}_{\bs h_a} (j,k) := r^{(p)}(j,k) \, \exp\{ \bs h_a(k) - \bs h_a(j)
\}$.
Since the chain $\bb X^{(p)}_t$ is irreducible on $S^a_p$, so is the
one with jump rates $r^{(p)}_{\bs h_a}(j,k)$. Hence, $\vartheta_{p}$ is the
unique stationary state. Since $\omega$ is also stationary,
$\vartheta_{p} (\,\cdot\,) = \Omega^{-1}_a \omega (\,\cdot\,)$.

This proves the uniqueness of limit points of the sequence
$\mu^{p,a}_n$ and completes the proof of the lemma.
\end{proof}

Fix $j\in S^{\rm sgl}_p$. By Lemma \ref{l35}, the Markov chain induced
by the generator $\mf R_{\ms V^{(p)}_j} \mf T_{\ms V^{(p)}} \ms L_n$
is irreducible. Denote by $\nu^{p,j}_n \in \ms P(\ms V^{(p)}_j)$ its
stationary state.

\begin{lemma}
\label{l23}
For each $j\in S^{\rm sgl}_p$, the sequence of probability measures
$\nu^{p,j}_n$ converges to $\pi^{(p)}_j$.
\end{lemma}

\begin{proof}
The proof is similar (and simpler) than the one of Lemma \ref{l12}. It
amounts to take $H=0$ and $\ms W^{(p)}_a = \ms V^{(p)}_j$.
\end{proof}

Consider the measures $\mu^{p,a}_n$, $\nu^{p,j}_n$, $1\le a\le m_p$, 
$j \in S^{\rm sgl}_p$, as probability measures on $\ms V^{(p)}$, and
let $\mu_n \in \ms P(\ms V^{(p)})$ be the measure given by
\begin{equation}
\label{26}
\mu_n(\,\cdot\,) \;=\; \sum_{a=1}^{m_p} \Omega_a \,
\mu^{p,a}_n (\,\cdot\,)\,
\;+\;
\sum_{j \in S^{\rm sgl}_p} \omega_j \, \nu^{p,j}_n (\,\cdot\,)  \;. 
\end{equation}
Next result follows from \eqref{27} and Lemmata \ref{l12} and
\ref{l23}.

\begin{corollary}
\label{l22}
The sequence of probability measures $\mu_n$ converges to $\mu$.
\end{corollary}

Fix $\mu\in \ms P(\ms V^{(p)})$ which can be represented as
$\mu = \sum_{j\in S_p} \omega_j \, \pi^{(p)}_j$ for some
$\omega \in \ms P (S_p)$ such that $\omega_j >0$ for all $j\in S_p$.
Recall the definition of $\Omega_a$, $\omega^{(a)}$ introduced in
Lemma \ref{l12}.  By definition of $\ms I^{(p)}$, Lemma \ref{l04} and
\eqref{05},
\begin{equation}
\label{18}
\ms I^{(p)} (\mu) \, =\, \sum_{a=1}^{m_p}
\Omega_a\, I_{\bb L^{(p)}_a} (\omega^{(a)}) 
\;+\;
\sum_{a=1}^{\ell_p} \sum_{b\neq a}
\sum_{j\in \mf Q^{(p)}_a} \sum_{k\in \mf Q^{(p)}_b} \omega_j \,
r^{(p)}(j,k)  \;.
\end{equation}
Note that we can restrict the first sum in the second term of the
right-hand side to the transient equivalent classes.

\begin{lemma}
\label{l11}
Fix $1\le p\le \mf q$, and a measure $\mu\in \ms P(\ms V^{(p)})$ which
can be represented as $\mu = \sum_{j\in S_p} \omega_j \, \pi^{(p)}_j$
for some $\omega \in \ms P (S_p)$ such that $\omega_j >0$ for all
$j\in S_p$. Let $(\mu_n : n\ge 1)$ be the sequence of probability
measures introduced in \eqref{26}. Then,
\begin{equation*}
\limsup_{n\to\infty} \theta^{(p)}_n\, \ms I^{(p)}_n(\mu_n) \;\le\; \ms
I^{(p)}(\mu)\;.  
\end{equation*}
\end{lemma}

\begin{proof}
By convexity,
\begin{equation}
\label{32}
\ms I^{(p)}_n(\mu_n) \;\le\;
\sum_{a=1}^{m_p} \Omega_a \,
\ms I^{(p)}_n( \mu^{p,a}_n)  \,+\,
\sum_{j \in S^{\rm sgl}_p} \omega_j \,
\ms I^{(p)}_n(\nu^{p,j}_n) \;. 
\end{equation}
We investigate the asymptotic behavior of each term separately.

Fix $j\in S^{\rm sgl}_p$. By Lemma \ref{l04} and equation
\eqref{05}, since the support of the measure $\nu^{p,j}_n$ is the set
$\ms V^{(p)}_j$, and, by Lemma \ref{l35}, the trace process
$Y^{n,p}_t$ reflected at $\ms V^{(p)}_j$ is irreducible,
\begin{equation*}
\ms I^{(p)}_n  (\nu^{p,j}_n)
\, =\, I_{\mf R_{\ms V^{(p)}_j} \mf T_{\ms V^{(p)}} \ms L_n} (\nu^{p,j}_n) 
\;+\;
\sum_{x\in \ms V^{(p)}_j}
\sum_{y\in \ms V^{(p)} \setminus \ms V^{(p)}_j}
\nu^{p,j}_n(x)\, R^{(p)}_n(x,y)  \;.
\end{equation*}
The first term on the right-hand side vanishes because $\nu^{p,j}_n$
is the stationary state of the process induced by the generator $\mf
R_{\ms V^{(p)}_j} \mf T_{\ms V^{(p)}} \ms L_n$. By the proof of
\eqref{16b}, or \eqref{31}, the
second term multiplied by $\theta^{(p)}_n$ converges to $\sum_{k\in
S_p \setminus \{j\}} r^{(p)}(j,k)$ so that
\begin{equation*}
\lim_{n\to\infty} 
\theta^{(p)}_n\,
\sum_{j \in S^{\rm sgl}_p} \omega_j \, \ms I^{(p)}_n(\nu^{p,j}_n)
\;=\; \sum_{j \in S^{\rm sgl}_p} 
\sum_{k\in S_p \setminus \{j\}} \omega_j\, r^{(p)}(j,k)\;.
\end{equation*}

The analysis of the asymptotic behavior of the first term on the
right-hand side in \eqref{32} is similar.  Fix $1\le a\le m_p$. Since
the support of the measure $\mu^{p,a}_n$ is the set $\ms W^{(p)}_a$,
and, by Lemma \ref{l36}, the trace process $Y^{n,p}_t$ reflected at
$\ms W^{(p)}_a$ is irreducible, by Lemma \eqref{l04} and equation
\eqref{05},
\begin{equation}
\label{34}
\ms I^{(p)}_n  (\mu^{p,a}_n)
\, =\, I_{\mf R_{\ms W^{(p)}_a} \mf T_{\ms V^{(p)}} \ms L_n}
(\mu^{p,a}_n) 
\;+\;
\sum_{x\in \ms W^{(p)}_a}
\sum_{y\in \ms V^{(p)} \setminus \ms W^{(p)}_a}
\mu^{p,a}_n(x)\, R^{(p)}_n(x,y)  \;.
\end{equation}
Since $\mu^{p,a}_n$ is the stationary state of the dynamics induced by
the generator $\mf M_{H_a}$
$\mf R_{\ms W^{(p)}_a} \mf T_{\ms V^{(p)}} \ms L_n$, where $H_a$ is
the function introduced in \eqref{33}, by Lemma \ref{l10},
\begin{equation*}
I_{\mf R_{\ms W^{(p)}_a} \mf T_{\ms V^{(p)}} \ms L_n}
(\mu^{p,a}_n) \;=\; -\,
\sum_{x\in \ms W^{(p)}_a} 
\sum_{y\in \ms W^{(p)}_a \setminus \{x\}}
\mu^{p,a}_n(x) \, R_n(x,y)\, 
\big[\, e^{H_a(y) -H_a(x)} - 1\,\big]\;.
\end{equation*}
As $H_a$ is constant and equal to $\bs h_a(j)$ on each set
$\ms V^{(p)}_j$, $j\in \mf D^{(p)}_a = S^a_p$, the previous expression
is equal to
\begin{equation*}
-\, \sum_{j\in S^a_p}
\sum_{k\in S^a_p \setminus \{j\}}
\big[\, e^{\bs h_a(k) - \bs h_a(j)} - 1\,\big]
\sum_{x\in \ms V^{(p)}_j} 
\sum_{y\in \ms V^{(p)}_k }
\mu^{p,a}_n(x) \, R_n(x,y) \;.
\end{equation*}
Hence, by the proof of \eqref{16b}, or \eqref{31},
\begin{equation*}
\lim_{n\to\infty} \theta^{(p)}_n\,
I_{\mf R_{\ms W^{(p)}_a} \mf T_{\ms V^{(p)}} \ms L_n}
(\mu^{p,a}_n) \;=\; -\,
\sum_{j\in S^a_p}
\sum_{k\in S^a_p \setminus \{j\}}
\big[\, e^{\bs h_a(k) - \bs h_a(j)} - 1\,\big]
\, \omega^{(a)}_j \, r^{(p)}(j,k) \;.
\end{equation*}
As $\bs h_a \colon S^a_p \to \bb R$ is the function which
solves the optimal problem \eqref{19} for
$I_{\bb L^{(p)}_a} (\omega^{(a)})$, the right-hand side is equal to 
$I_{\bb L^{(p)}_a} (\omega^{(a)})$ so that
\begin{equation*}
\lim_{n\to\infty} \theta^{(p)}_n\,
I_{\mf R_{\ms W^{(p)}_a} \mf T_{\ms V^{(p)}} \ms L_n}
(\mu^{p,a}_n) \;=\; I_{\bb L^{(p)}_a} (\omega^{(a)}) \;.
\end{equation*}

By similar reasons, the second term on the right-hand in \eqref{34}
multiplied by $\theta^{(p)}_n$ converges to
$\sum_{j\in S^a_p} \sum_{k\in S_p \setminus S^a_p}
\omega^{(a)}_j \, r^{(p)}(j,k)$ so that
\begin{equation*}
\lim_{n\to\infty} 
\theta^{(p)}_n\,
\sum_{a=1}^{m_p} \Omega_a \,
\ms I^{(p)}_n( \mu^{p,a}_n) 
\;=\; \sum_{a=1}^{m_p} \Omega_a
I_{\bb L^{(p)}_a} (\omega^{(a)}) \;+\;
\sum_{a=1}^{m_p} 
\sum_{j\in S^a_p} \sum_{k\in S_p \setminus S^a_p}
\omega_j \, r^{(p)}(j,k) \;.
\end{equation*}
By \eqref{18}, the right-hand side is equal to $\ms I^{(p)}(\mu)$,
completing the proof of the lemma.
\end{proof}

\begin{proof}[Proof of Proposition \ref{l16}]
Fix $1\le p\le \mf q$ and $\mu\in \ms P(\ms V^{(p)})$.  By Lemmata
\ref{l24} and \ref{l15}, we may assume that $\mu(x)>0$ for all
$x\in \ms V^{(p)}$. If $\ms I^{(p)}(\mu) = \infty$, there is nothing
to prove. Assume, therefore, that
$\mu = \sum_{j\in S_p} \omega_j \, \pi^{(p)}_j$ for some
$\omega \in \ms P(S_p)$. Since $\mu(x)>0$ for all $x\in \ms V^{(p)}$,
$\omega_j>0$ for all $j\in S_p$.  To complete the proof it remains to
recall the assertions of Corollary \ref{l22} and Lemma \ref{l11}.
\end{proof}

\section{Proofs of Proposition \ref{mt2} and Theorem \ref{mt1}}
\label{sec4}

We first present some properties of the functionals $\ms I^{(p)}$.

\begin{lemma}
\label{l39}
Fix $1\le p < \mf q$. Then,
\begin{equation*}
\ms I^{(p+1)} (\mu) \;<\; \infty \quad \text{if and only if}\quad
\ms I^{(p)} (\mu) \;=\; 0\;.
\end{equation*}
\end{lemma}

\begin{proof}
Suppose that $\ms I^{(p)} (\mu) = 0$. Then, by \eqref{o-83b}, $\mu =
\sum_{j\in S_p} \omega_j\, \pi^{(p)}_j$ for some $\omega\in \ms
P(S_p)$ and $\bb I^{(p)}(\omega)=0$. By the definition \eqref{40} of
$\bb I^{(p)}$ and Lemma \ref{l32}, $\omega$ is a stationary state of
the Markov chain $\bb X^{(p)}_t$, that is, $\omega$ is a convex
combination of the measures $M^{(p)}_m$, $m\in S_{p+1}$:
\begin{equation*}
\omega(j) \;=\; \sum_{m\in S_{p+1}} \vartheta (m)\, M^{(p)}_m(j)\;,
\quad j\,\in\,S_p\;,
\end{equation*}
for some $\vartheta \in \ms P(S_{p+1})$.  Inserting this expression in
the formula for $\mu$ and changing the order of summation yields that
\begin{equation*}
\mu \;=\; \sum_{m\in S_{p+1}} \vartheta (m) \, \sum_{j\in S_p} M^{(p)}_m(j)
\, \pi^{(p)}_j \;=\;
\sum_{m\in S_{p+1}} \vartheta (m) \, \pi^{(p+1)}_m\;,
\end{equation*}
where we used identity \eqref{o-80} in the last step. This proves the
first assertion of the lemma because
$\bb I^{(p+1)} (\vartheta) <\infty$ for all
$\vartheta\in \ms P(S_{p+1})$. We turn to the converse.

Suppose that $\ms I^{(p+1)} (\mu) < \infty$. In this case, by
\eqref{o-83b},
$\mu = \sum_{m\in S_{p+1}} \vartheta (m) \, \pi^{(p+1)}_m$ for some
$\vartheta \in \ms P(S_{p+1})$. By \eqref{o-80}, this identity can be
rewritten as
\begin{equation*}
\mu(\,\cdot\,) \;=\; \sum_{j\in S_p} \Big( \, \sum_{m\in S_{p+1}}
\vartheta (m) \, M^{(p)}_m(j)\,\Big)
\, \pi^{(p)}_j (\,\cdot\,) \;.
\end{equation*}
Therefore, by definition of $\ms I^{(p)}$,
$\ms I^{(p)}(\mu) = \bb I^{(p)}(\omega)$, where
$\omega(j) = \sum_{m\in S_{p+1}} \vartheta (m) \, M^{(p)}_m(j)$.
As the measures $M^{(p)}_m$ are stationary for the chain $\bb
X^{(p)}_t$, so is $\omega$. Thus, by Lemma \ref{l32}, $\bb
I^{(p)}(\omega)=0$, as claimed.
\end{proof}

By Lemma \ref{l32}, $\ms I^{(0)} (\mu) \,=\, 0$ if and only if there
exists a probability measure $\omega$ on $S_1$ such that
\begin{equation*}
\mu \;=\; \sum_{j\in S_1} \omega_j\, \pi^{(1)}_j\;.
\end{equation*}
By \eqref{o-83b} and since $\bb I^{(1)}(\omega) <\infty$ for all
$\omega \in \ms P(S_1)$, $\mu$ has this form if and only if
$\ms I^{(1)}(\mu)<\infty$. Hence, the previous lemma holds for $p=0$
as well:
\begin{equation}
\label{o-79}
\ms I^{(1)} (\mu) \;<\; \infty \quad \text{if and only if}\quad
\ms I^{(0)} (\mu) \;=\; 0\;.
\end{equation}

We turn to the proof of Proposition \ref{mt2}. In sake of
completeness, we reproduce the proof that $\ms I^{(0)}$ is a
$\Gamma-\liminf$ of the sequence $\ms I_n$ presented in \cite{bgl2}
and which applies to non-reversible dynamics. Next result is the
second half of \cite[Proposition 8.3]{bgl2}.

\begin{lemma}
\label{l37}
The functional $\ms I^{(0)} \colon \ms P(V) \to \bb R_+$ is a
$\Gamma-\liminf$ of the sequence $\ms I_n$.
\end{lemma}

\begin{proof}
Fix  $\mu\in \ms P(V)$ and a sequence
of probability measures $\mu_n$ in $\ms P(\ms V)$ converging to $\mu$.
By definition of $\ms I_n$,
\begin{equation*}
\ms I_n(\mu_n) \; \ge \; -\, \int_V \frac{\ms L_n u}{u}\, d\mu_n
\;=\; -\, \sum_{x\in V} \frac{\mu_n(x)}{u(x)}\,
\sum_{y\in V} R_n(x,y) \, [\, u(y) - u(x)\,]
\end{equation*}
for all $u: V \to (0,\infty)$. As $\mu_n\to \mu$ and $R_n \to \bb
R_0$, this expression converges to
\begin{equation*}
-\, \sum_{y\neq x\in V} \frac{\mu(x)}{u(x)}\,
\bb R_0(x,y) \, [\, u(y) - u(x)\,] \;.
\end{equation*}
Therefore,
\begin{equation*}
\liminf_{n\to\infty} \ms I_n(\mu_n) \; \ge \;
\sup_{u>0}  \, -\, \sum_{y\neq x\in V} \frac{\mu(x)}{u(x)}\,
\bb R_0(x,y) \, [\, u(y) - u(x)\,] \;=\; \ms I^{(0)}(\mu) \;,
\end{equation*}
which completes the proof of the lemma.
\end{proof}

We turn to the $\Gamma-\limsup$. Fix a measure $\mu \in \ms P(V)$ such
that $\mu(x)>0$ for all $x\in V$.  Denote by
$\color{blue} \mf D^{(0)}_1, \dots, \mf D^{(0)}_{m_0}$ the equivalent
classes of the chain $\bb X_t$ which are not singletons.

By definition, the Markov chain $\bb X_t$ reflected at
$\mf D^{(0)}_a$, $1\le a\le m_0$, is irreducible. Denote by
$\mu_{\mf D^{(0)}_a}$ the measure $\mu$ conditioned to $\mf D^{(0)}_a$
defined by equation \eqref{57}. Let
$H_a \colon \mf D^{(0)}_a \to \bb R$ be the function given by Lemma
\ref{l02} which turns $\mu_{\mf D^{(0)}_a}$ a stationary state for the
Markov chain induced by
$\mf M_{H_a} \mf R_{\mf D^{(0)}_a}\, \bb L^{(0)}$, the generator
$\mf R_{\mf D^{(0)}_a}\, \bb L^{(0)}$ tilted by $H_a$.

By Corollary \ref{l38}, 
\begin{equation}
\label{12}
\ms I^{(0)} (\mu) \;=\;
-\, \sum_{a=1}^{m_0} \sum_{x\in \ms D_a} \sum_{y\in \ms
D_a\setminus \{x\}}
\mu(x) \, \bb R_{H_a} (x,y) 
\;+\;
\sum_{x\in V} \sum_{y\in V\setminus \{x\}} \mu(x)\, \bb R_0(x,y)\;,
\end{equation}
where $\bb R_{H_a} (x,y) = \bb R_{0} (x,y) \, e^{H_a(y) - H_a(x)} $

\begin{lemma}
\label{l09}
For all $\mu\in \ms P(V)$,
\begin{equation*}
\limsup_{n\to\infty} \ms I_n(\mu) \;\le\; \ms I^{(0)}(\mu)\;. 
\end{equation*}
\end{lemma}

\begin{proof}
Denote by $X^{\mu,n}_t$ the Markov chain $X^n_t$ reflected at $V_\mu$,
and by $\ms D^n_1, \dots, \ms D^n_{m_n}$ the equivalent classes of the
chain $X^{\mu,n}_t$ which are not singletons. By assumption
\eqref{51}, the sets $\ms D_a$ do not depend on $n$ and we may remove
the index $n$ from the notation. Moreover, since $R_n(x,y)$ converges
to $\bb R_0(x,y)$, for all $1\le a\le m_0$, there exists $1\le b\le m$
such that $\ms D^{(0)}_a \subset \ms D_b$.

By Lemma \ref{l04}, $\ms I_n(\mu) = \ms K_n(\mu)$, where
$\ms K_n(\mu)$ is given by \eqref{05b} with the rates $R$ replaced by
$R_n$. The functional $\ms K_n$ is composed of two terms. The second,
as $n\to \infty$, converges to
\begin{equation*}
\sum_{x\in V} \sum_{y\in V\setminus \{x\}} \mu(x)\, \bb R_0(x,y) \;. 
\end{equation*}

We turn to the first term of $\ms K_n$, given by
\begin{equation}
\label{43}
-\, \sum_{b=1}^{m}
\sum_{x\in \ms D_b} \sum_{y\in \ms D_b\setminus \{x\}}
\mu(x) \, R_n(x,y) \,  e^{H^{(b)}_n(y) - H^{(b)}_n(x)} \;, 
\end{equation}
where $H^{(b)}_n \colon \ms D_b \to \bb R$ is the function (unique up
to an additive constant) which turns $\mu$ a stationary state for the
chain induced by $\mf M_{H^{(b)}_n} \mf R_{\ms D_b}\, \ms L_n$ (the
generator $\mf R_{\ms D_b}\, \ms L_n$ tilted by $H^{(b)}_n$).

Since for all $1\le a\le m_0$, there exists $1\le b\le m$ such that
$\ms D^{(0)}_a \subset \ms D_b$, the sum appearing in the previous
displayed equation is bounded above by
\begin{equation}
\label{13}
-\, \sum_{b=1}^{m} \sum_a
\sum_{x\in \ms D^{(0)}_a} \sum_{y\in \ms D^{(0)}_a\setminus \{x\}}
\mu(x) \, R_n(x,y) \,  e^{H^{(b)}_n(y) - H^{(b)}_n(x)} \;,
\end{equation}
where the second sum is performed over all $1\le a\le m_0$ such that
$\ms D^{(0)}_a \subset \ms D_b$.

Fix $b$ and $a$ satisfying $\ms D^{(0)}_a \subset \ms D_b$.  By
\eqref{38}, there exists a finite constant $C^{(b)}_n$ such that
$|\, H^{(b)}_n(y) - H^{(b)}_n(x) \,| \le C^{(b)}_n$ for all
$y\not = x \in \ms D_b$. Since $R_n(x,y) \to \bb R_0(x,y)$, and
$\mu(x) >0$ for all $x \in \ms D^{(0)}_a$, as $\bb X^\mu_t$ is
irreducible in $\ms D^{(0)}_a$, by Remark \ref{rm1}, there exists a
finite constant $C_a$, independent of $n$, such that
$|\, H^{(b)}_n(y) - H^{(b)}_n(x) \,| \le C_a$ for all
$y\not = x \in \ms D^{(0)}_a$. Therefore, there exists a function
$G_a\colon \ms D^{(0)}_a \to \bb R$ and a subsequence $n'$ such that
$H_{n'} (y) - H_{n'}(x) \to G(y) - G(x)$ for all $x$,
$y\in \ms D^{(0)}_a$.

In conclusion, through a subsequence, \eqref{13} converges to
\begin{equation*}
\begin{aligned}
& -\, \sum_{a=1}^{m_0}
\sum_{x\in \ms D^{(0)}_a} \sum_{y\in \ms D^{(0)}_a\setminus \{x\}}
\mu(x) \, \bb R_0 (x,y) \,  e^{G_a(y) - G_a(x)} \\
&\quad \le\; -\, 
\sum_{a=1}^{m_0} \sum_{x\in \ms D^{(0)}_a}
\sum_{y\in \ms D^{(0)}_a\setminus \{x\}}
\mu(x) \, \bb R_0 (x,y) \,  e^{H_a(y) - H_a(x)} \;,
\end{aligned}
\end{equation*}
where $H_a$ is the function which appears in \eqref{12}. The
inequality holds because for each $1\le a\le m_0$, $H_a$ is the
function which optimises the sum. To complete the proof of the lemma,
it remains to collect the previous estimates.
\end{proof}

Next result is a consequence of the two previous lemmata. 

\begin{corollary}
\label{l31}
The functional $\ms I_n$ $\Gamma$-converge to $\ms I^{(0)}$.
\end{corollary}

\subsection*{Proof of Theorem \ref{mt1}}

The proof is by induction in $p$. The case $p=0$ is covered by
Corollary \ref{l31}. Fix $1\le p\le \mf q$ and assume that the result
holds for $0\le p'<p$. In sake of completeness we reproduce below the
proof of the $\Gamma-\liminf$ taken from \cite[Proposition
8.4]{bgl2}.

\smallskip\noindent{\bf $\Gamma-\liminf$:} For all $\mu\in \ms P(V)$
and all sequence of probability measures $\mu_n \in \ms P(V)$ such
that $\mu_n\to \mu$,
\begin{equation}
\label{44}
\liminf_{n\to\infty} \theta^{(p)}_n\, \ms I^{(p)}_n(\mu_n) \;\ge\; \ms
I^{(p)}(\mu)\;.  
\end{equation}

Fix a probability measure
$\mu$ on $V$ and a sequence $\mu_n$ converging to $\mu$.  Suppose that
$\ms I^{(p-1)} (\mu)>0$. In this case, since
$\theta^{(p-1)}_n\, \ms I_n$ $\Gamma$-converges to $\ms I^{(p-1)}$ and
$\theta^{(p)}_n/\theta^{(p-1)}_n \to\infty$,
\begin{equation*}
\liminf_{n\to\infty} \theta^{(p)}_n\, \ms I_n(\mu_n) \;=\;
\liminf_{n\to\infty} \frac{\theta^{(p)}_n}{\theta^{(p-1)}_n}
\, \theta^{(p-1)}_n\, \ms I_n(\mu_n) \;\ge\;
\ms I^{(p-1)} (\mu)\,
\lim_{n\to\infty} \frac{\theta^{(p)}_n}{\theta^{(p-1)}_n}
\;=\; \infty\;.
\end{equation*}
On the other hand, by Lemma \ref{l39}, $\ms I^{(p)} (\mu) =
\infty$. This proves the $\Gamma-\liminf$ convergence for measures
$\mu$ such that $\ms I^{(p-1)} (\mu)>0$.

Assume that $\ms I^{(p-1)} (\mu)=0$. By Lemma \ref{l39} and
\eqref{o-79}, there exists a probability measure $\omega$ on $S_p$
such that $\mu = \sum_{j\in S_p} \omega_j\, \pi^{(p)}_j$. By
definition of $\ms I_n$,
\begin{equation*}
\ms I_n(\mu_n) \; \ge \; -\, \int_V \frac{\ms L_n u}{u}\, d\mu_n
\end{equation*}
for all $u: V \to (0,\infty)$.

Fix a function $h: \ms V^{(p)} \to (0,\infty)$ which is constant on
each $\ms V^{(p)}_j$, $j\in S_p$:
$h = \sum_{j\in S_p} \mb h(j) \, \chi_{\ms V^{(p)}_j}$. Let
$u_n\colon V\to \bb R$ be the solution of the Poisson equation
\eqref{07} with $\ms L = \ms L_n$, $\ms A = \ms V^{(p)}$ and $u = h$. By
the representation \eqref{71}, it is clear that
$u_n(x) \in (0,\infty)$ for all $x\in V$.

Since $u_n$ is harmonic on $V \setminus \ms V^{(p)}$ and $u_n = h$ on
$\ms V^{(p)}$, by \eqref{09}, the right-hand side of the previous
displayed equation with $u=u_n$ is equal to
\begin{equation*}
-\, \int_{\ms V^{(p)}} \frac{\ms L_n u_n}{u_n}\, d\mu_n
\;=\; -\, \int_{\ms V^{(p)}} \frac{\ms L_n u_n}{h}\, d\mu_n
\;=\; -\, \int_{\ms V^{(p)}}
\frac{(\mf T_{\ms V^{(p)}} \ms L_n)\, h}{h}\, d\mu_n \;.
\end{equation*}
Since $h$ is constant on each set $\ms V^{(p)}_j$ (and equal to $\mb
h(j)$), the last integral is equal to
\begin{equation*}
-\, \sum_{j,k\in S_p} \frac{[\, \mb h(k) - \mb h(j)\,]}{\mb h(j)}\,
\sum_{x\in \ms V^{(p)}_j} \pi_n(x)\, \frac{\mu_n (x)}{\pi_n(x)}\,
R^{(p)}_n(x, \ms V^{(p)}_k)\;,
\end{equation*}
where
$R^{(p)}_n(x, \ms V^{(p)}_k) = \sum_{y\in \ms V^{(p)}_k} R^{(p)}_n(x,
y)$. By \eqref{o-58},
$\pi_n(x)/\pi_n(\ms V^{(p)}_j) \to \pi^{(p)}_j(x)$ for all
$x\in \ms V^{(p)}_j$. Thus, since
$\mu_n \to \mu = \sum_{j\in S_p} \omega_j\, \pi^{(p)}_j$,
\begin{equation*}
\lim_{n\to \infty} \pi_n(\ms V^{(p)}_j)\, \frac{\mu_n (x)}{\pi_n(x)}
\;=\; \omega_j \quad\text{for all $x\in \ms V^{(p)}_j$} \;.
\end{equation*}
Therefore, by \eqref{20}, \eqref{o-34}, as $n\to\infty$, the
penultimate expression multiplied by $\theta^{(p)}_n $ converges to
\begin{equation*}
-\, \sum_{j\in S_p} \omega_j \, \frac{1}{\mb h(j)}\,
\sum_{k\in S_p} r^{(p)}(j,k) \, [\, \mb h(k) - \mb h(j)\,]
\;=\; -\, \sum_{j\in S_p} \omega_j \, \frac{\bb L^{(p)} \mb h}{\mb h} \;.
\end{equation*}

Summarising, we proved that
\begin{equation*}
\liminf_{n\to\infty} \theta^{(p)}_n\, \ms I_n(\mu_n) \; \ge \;
\sup_{\mb h} \,  -\, \sum_{j\in S_p} \omega_j \, \frac{\bb L^{(p)} \mb
h}{\mb h} \;,
\end{equation*}
where the supremum is carried over all functions
$\mb h: S_p \to (0,\infty)$. By \eqref{40}, \eqref{o-83b}, the
right-hand side is precisely $\ms I^{(p)}(\mu)$, which completes the
proof of the $\Gamma-\liminf$. 

\smallskip\noindent{$\bs \Gamma-\limsup$.}  Fix
$\mu\in \ms P(V)$.  If $\ms I^{(p)}(\mu) = \infty$, there is
nothing to prove. Assume, therefore, that
$\mu = \sum_{j\in S_p} \omega_j \, \pi^{(p)}_j$ for some
$\omega \in \ms P(S_p)$.

By Lemmata \ref{l24} and \ref{l25}, it is enough to prove the theorem
for measures $\mu = \sum_{j\in S_p} \omega_j \, \pi^{(p)}_j$ for some
$\omega \in \ms P(S_p)$ such that $\omega_j>0$ for all $j\in S_p$.
Fix such a measure $\mu$.  Let $\mu_n\in \ms P(\ms V^{(p)})$ be the
measure given by \eqref{26}. By Corollary \ref{l22} and Lemma
\ref{l11}, $\mu_n\to \mu$ and
\begin{equation}
\label{30c}
\limsup_{n\to\infty} \theta^{(p)}_n\, \ms I^{(p)}_n(\mu_n) \;\le\; \ms
I^{(p)}(\mu)\;.  
\end{equation}

Since the trace process $Y^{n,p}_t$ is irreducible and $\mu_n(x)>0$
for all $x\in \ms V^{(p)}$, by Lemma \ref{l02}, there exists
$u_n\colon \ms V^{(p)} \to (0,\infty)$ such that
\begin{equation*}
\ms I^{(p)}_n(\mu_n) \;=\;  - \,
\int_{\ms V^{(p)}} \frac{1}{u_n}\, 
\big [ \, (\, \mf T_{\ms V^{(p)}} \ms L_n)\,  u_n \,\big]\,
\; d \mu_n \;.
\end{equation*}
Denote by $v_n$ the harmonic extension of $u_n$ to $V$ given by
\eqref{07} with $\ms A$, $\ms L$ replaced by $\ms V^{(p)}$, $\ms L_n$,
respectively. Let $\nu_n$ be the stationary state of the tilted
generator $\mf M_{v_n}\, \ms L_n$. By Proposition \ref{l05},
\begin{equation}
\label{30b}
\ms I_n(\nu_n) \;\le\; \ms I^{(p)}_n (\mu_n)\;.  
\end{equation}
In view of \eqref{30c}, \eqref{30b}, it remains to show that
$\nu_n \to \mu$.

As $\nu_n$ is the stationary state of the Markov chain $X^{(n)}_t$
tilted by $v_n$, by \cite[Proposition 6.3]{bl2}, $\nu_n$ conditioned
to $\ms V^{(p)}$ is the stationary state of the Markov chain induced
by the generator $\mf T_{\ms V^{(p)}}\, \mf M_{v_n}\, \ms L_n$. By
Lemma \ref{l06} this generator coincides with
$\mf M_{u_n}\, \mf T_{\ms V^{(p)}}\,\ms L_n$. By definition, $\mu_n$
is the stationary state of this later Markov chain. Hence,
$\mu_n (\,\cdot\,) = \nu_n (\,\cdot\,|\, \ms V^{(p)}\,)$.

Since $\mu_n \to \mu$ and
$\mu_n (\,\cdot\,) = \nu_n (\,\cdot\,|\, \ms V^{(p)}\,)$, it is enough
to show that $\nu_n (\,\ms V^{(p)}\,) \to 1$.  Assume, by
contradiction, that $\limsup_n \nu_n(z) >0$ for some
$z\in V \setminus \ms V^{(p)}$. Since $\ms P(V)$ is compact for the
weak topology, consider a subsequence, still denoted by $\nu_n$, such
that $\nu_n \to \nu \in \ms P(V)$, $\nu(z)>0$. By the
$\Gamma-\liminf$,
\begin{equation*}
\liminf_{n\to\infty} \theta^{(p)}_n\, \ms I_n(\nu_n) \;\ge\; \ms
I^{(p)}(\nu)\;. 
\end{equation*}
Since $\nu(z)>0$, $z\in V\setminus \ms V^{(p)}$,
$\ms I^{(p)}(\nu) = +\infty$. However, by \eqref{30c}, \eqref{30b},
\begin{equation*}
\limsup_{n\to\infty} \theta^{(p)}_n\, \ms I_n(\nu_n) \;\le\;
\ms I^{(p)}(\mu) \;=\; \bb I^{(p)} (\omega) \;<\; \infty \;.
\end{equation*}
Hence, $\nu_n (\,\ms V^{(p)}\,) \to 1$ and $\nu_n \to \mu$, which
completes the proof of the theorem.  \qed

\appendix

\section{The rate function}
\label{sec01}

Fix a finite set $V$. Consider a $V$-valued continuous-time Markov
chain $\color{blue} (X_t: t\ge 0)$, and denote by $\ms L$ its
generator. The jump rates are represented by
$\color{blue} R(\,\cdot\,,\,\cdot\,)$, so that
\begin{equation}
\label{36}
{\color{blue} (\ms L f)(x)}
\;=\; \sum_{y\in V} R(x,y)\, \{\,
f(y)\,-\, f(x)\,\}
\end{equation}
for all functions $f\colon V\to \bb R$.  Denote by
$\color{blue} \lambda(x)$, $x\in V$, the holding rates and by
$\color{blue} p(x,y)$, $x$, $y\in V$, the jump probabilities, so that
$R(x,y) = \lambda (x) \, p(x,y)$. Mind that we do not suppose the
process to be irreducible.  We assume, however, that $\lambda(x)>0$
for all $x\in V$. The case where some holding rates might vanish is
considered at the end of this section.

Denote by $\color{blue} \ms P (V)$ the space of probability measures
on $V$ endowed with the weak topology. For a function
$H\colon V \to \bb R$, define $J_H \colon \ms P (V) \to \bb R$ by
\begin{equation}
\label{37}
{\color{blue} J_H(\mu)} \; : =\; -\, \int_V e^{-H} \ms L e^H\, d\mu
\;=\; -\, \sum_{x, y\in V} \mu(x) \, R(x,y) \, \big[\, e^{H(y) - H(x)}
-1 \,\big] \;,
\end{equation}
and let
\begin{equation}
\label{19}
{\color{blue} I(\mu)} \;: =\; \sup_{H} J_H(\mu)\;,
\end{equation}
where the supremum is carried over all functions
$H\colon V \to \bb R$.

To stress the dependence of the functionals $J_H$ and $I$ on the
generator $\ms L$, we sometimes denote them by
$\color{blue} J_{\ms L, H}$ and $\color{blue} I_{\ms L}$,
respectively.  Next result collects simple properties of the
functionals $J_H$ and $I$.

\begin{lemma}
\label{l01}
For each $\mu\in \ms P (V)$, the functional $H\mapsto J_H(\mu)$ is
concave, and $J_{H+c} (\mu) = J_H(\mu)$ for all constants $c$.  The
functional $I$ is convex, lower-semi\-con\-tin\-uous, non-negative,
and bounded by $\sum_{x\in V} \mu(x) \, \lambda(x)$.
\end{lemma}

Recall that a measure $\mu\in \ms P(V)$ is a stationary state for the
Markov chain induced by the generator $\ms L$ if
$\int_V (\ms L f) \, d\mu =0$ for all function $f\colon V\to \bb R$.
As we do not assume the chain to be irreducible, the stationary state
may not be unique.

The Euler-Lagrange equation for $I$ reads as
\begin{equation}
\label{11}
\int_V (\mf M_H\, \ms L) \, G \, d\mu \;=\; 0 \quad \text{for all}\;\;
G\colon V \to \bb R\;.
\end{equation}
In this formula, for a function $H\colon V\to \bb R$,
$\mf M_H\, \ms L$ represents the tilted generator given by
\begin{equation}
\label{56}
{\color{blue} [\,(\mf M_H\, \ms L) \, f\, ]\, (x)} \;=\;
\sum_{y\in V}
e^{- H(x)} \, R(x,y) \, e^{H(y)}  \,  \big[\, f(y) \,-\, f(x)
\,\big] 
\end{equation}
for $f\colon V \to \bb R$. Let
$\color{blue} R_H (x,y)\, :=\, e^{- H(x)} \, R(x,y) \, e^{H(y)} $.
Next result clarify the meaning of the Euler-Lagrange equation
\eqref{11}. 

\begin{lemma}
\label{l10}
A probability measure $\mu$ in $V$ is a stationary state for the
Markov chain induced by the generator $\mf M_H\ms L$ if, and only if,
\begin{equation}
\label{08}
I(\mu) \;=\; J_H(\mu)\;.
\end{equation}
\end{lemma}

\begin{proof}
Suppose that $\mu$ is a stationary state for the Markov chain induced
by the generator $\mf M_H \ms L$. Then, for all functions
$G\colon V\to \bb R$,
\begin{equation*}
\begin{aligned}
J_G(\mu) \;-\; J_H(\mu) \; & =\;
-\, \sum_{x, y\in V} \mu(x)\, R (x,y)  \,
\big[\, e^{G(y) \,-\, G(x)} \,-\, e^{H(y) \,-\, H(x)} \,\big] \\
& =\;
-\, \sum_{x, y\in V} \mu(x)\, R_H (x,y)  \,
\big[\, e^{F(y) \,-\, F(x)} \,-\, 1 \,\big] \;,
\end{aligned}
\end{equation*}
where $F= G-H$. Since $\mu$ is a stationary state for the Markov chain
induced by the generator $\mf M_H \ms L$, the previous expression is equal
to
\begin{equation*}
-\, \sum_{x, y\in V} \mu(x)\, R_H (x,y)  \,
\big\{\, e^{F(y) \,-\, F(x)} \,-\, 1 \,-\, [\, F(y) \,-\, F(x)\,]
\,\big\}\;.
\end{equation*}
This expression is negative because $a\mapsto e^a - 1 - a$ is
positive. This proves that $I(\mu) = \sup_G J_G(\mu) \le J_H(\mu)$, as
claimed.

Conversely, suppose that \eqref{08} holds. Fix $G:V\to \bb R$. Then,
the function $a\mapsto J_{H + a G} (\mu)$ assumes a maximum at
$a=0$. Its derivative at $a=0$ is given by
$\int_V (\mf M_H \ms L) \, G\, d\mu$. Hence, \eqref{11} holds for all
$G$ yielding that $\mu$ is stationary for the Markov chain induced by
the generator $\mf M_H \ms L$.
\end{proof}

\begin{lemma}
\label{l02}
Assume that the Markov chain induced by the generator $\ms L$ is
irreducible and that $\mu(x)>0$ for all $x\in V$. Then, there exists a
function $H\colon V\to \bb R$, unique up to an additive constant, such
that
\begin{equation*}
I(\mu) \;=\; J_{H}(\mu)\;.
\end{equation*}
Moreover,
\begin{equation}
\label{45}
I(\mu) \;=\; \sum_{x, y\in V} \mu(x) \, R(x,y) \,
\Big\{\, [\, H(y) - H(x) \, ] \, e^{H(y) - H(x)}
\,-\, e^{H(y) - H(x)} \,+\, 1\, \Big\} \;,
\end{equation}
and
\begin{equation}
\label{38}
\max_{x,y\in V} |\, H (y) - H(x)\,|\, \le\,
|V|\, \ln\, \frac{1 \,+\, \sum_{x, y\in V} \mu(x) \, R(x,y)}
{\min_{z,w} \mu(z)\, R(z,y)} \;,
\end{equation}
where the minimum is performed over all edges $(z,w)$ such that
$R(z,w)>0$. 
\end{lemma}

\begin{proof}
Since $I(\mu)$ is bounded, there exists a sequence
$(H_n \colon n\ge 1)$ of functions $H_n \colon V\to \bb R$ such that
\begin{equation*}
I(\mu) \;=\; \lim_{n\to \infty} J_{H_n}(\mu)\;.
\end{equation*}
Fix $x_0 \in V$. Since $J_{H+c} (\mu) = J_H(\mu)$, redefine the
sequence $H_n$ so that $H_n(x_0) =0$ for all $n\ge 1$.

\smallskip\noindent{\it Claim 1:} The sequence $H_n$ is uniformly
bounded.

By definition of the sequence $H_n$ and since $I$ is positive, there
exists $n_0\ge 1$ such that
\begin{equation*}
\sum_{x, y\in V} \mu(x) \, R(x,y) \, \big[\, e^{H_n(y) - H_n(x)}
-1 \,\big] \;\le\;  1
\end{equation*}
for all $n\ge n_0$. Hence,
\begin{equation*}
\sum_{x, y\in V} \mu(x) \, R(x,y) \,e^{H_n(y) - H_n(x)}
\;\le\; C_0 \;: =\;  1 \,+\,
\sum_{x, y\in V} \mu(x) \, R(x,y) \;.
\end{equation*}
Thus, 
\begin{equation}
\label{01}
H_n(y) \,-\, H_n(x) \;\le\; C_1\;:=\;
\ln\, \frac{C_0}{\min_{x,y} \mu(x)\, R(x,y)} 
\end{equation}
for all $n\ge n_0$ and all edges $(x,y)$ such that $R(x,y)>0$.  In
this equation the minimum is performed over all edges $(x,y)$ such
that $\mu(x)\, R(x,y) >0$.

Fix $x\in V\setminus \{x_0\}$. Since the process is irreducible, there
exists a self-avoiding path $x_0, x_1, \dots, x_k=x$ such that
$R(x_i, x_{i+1})>0$ for all $0\le i<k$. Hence, by \eqref{01}, $H_n(x)
= H_n(x_k) - H_n(x_0) \le C_1 \, k \,\le\, C_2 \,:=\, C_1 \, |V|$.

Conversely, there exists a self-avoiding path
$x=y_0, y_1, \dots, y_j=x_0$ such that $R(y_i, y_{i+1})>0$ for all
$0\le i<j$. Hence, by \eqref{01},
$-H_n(x) = H_n(x_0) -H_n(x) = H_n(y_j) - H_n(y_0) \le C_1 \, j \le C_1
\, |V| \,=\, C_2$, which proves Claim 1.

As the sequence $H_n$ is uniformly bounded, we may extract a
subsequence, still denoted by $H_n$, which converges pointwisely to
a function $H$. By definition of the sequence $H_n$ and by continuity, 
\begin{equation*}
I(\mu) \;=\; \lim_{n\to\infty} J_{H_n}(\mu) \;=\; J_{H}(\mu)\;,
\end{equation*}
as asserted. Moreover,
$\max_{x,y\in V} |\, H (y) - H(x)\,|\, \le\, C_2$, proving \eqref{38}.

We turn to the proof of uniqueness. Assume that there are two
functions, denoted by $H$ and $G$, which minimize. Let
$F_\theta = \theta H + (1-\theta) G$, $0\le \theta \le 1$.  By
concavity of the functional $J$, for all $0\le \theta \le 1$,
\begin{equation*}
I(\mu) \;\ge\; J_{F_\theta}(\mu) \;\ge\;
\theta \, J_{H}(\mu) \;+\;  (1-\theta) \, J_{G}(\mu) \;=\;
I(\mu) \; .
\end{equation*}
Hence, $\theta \mapsto J_{F_\theta}(\mu)$ is constant. Taking the
second derivative yields that
\begin{equation*}
\sum_{x, y\in V} \mu(x) \, R(x,y) \,
\big\{ \, \big[\, G(y) - G (x) \,\big] \,-\,
\big[\, H (y) - H (x)\,\big]  \,\big\}^2
e^{F_\theta (y) - F_\theta (x)} \;=\; 0
\end{equation*}
for all $0<\theta<1$.  Hence, $G(y) - G (x) \,=\, H (y) - H (x)$ if
$\mu(x)\, R(x, y)>0$. As the process is irreducible and the measure
positive, $G= H + c$ for some constant $c\in\bb R$.

To show the validity of \eqref{45}, note that
\begin{equation*}
I(\mu) \;=\; J_{H}(\mu) \;=\;
\sum_{x, y\in V} \mu(x) \, R(x,y) \,
\big\{\, 1 \,-\, e^{H(y) - H(x)} \, \big\}\;.
\end{equation*}
Since, by Lemma \ref{l10}, $\mu$ is a stationary state for the
Markov chain induced by the generator $\mf M_H\ms L$,
\begin{equation*}
0 \;=\;
\sum_{x, y\in V} \mu(x) \, R(x,y) \, e^{H(y) - H(x)}\,
[\, H(y) - H(x) \, ] \;.
\end{equation*}
Adding the two previous identities yields \eqref{45}.
\end{proof}

\begin{remark}
\label{rm1}
It follows from the previous proof that the minimum in the denominator
in equation \eqref{38} can be restricted to a subset of edges $E_0$
which keeps the chain irreducible.
\end{remark}

Next result follows from the two previous lemmata.

\begin{corollary}
\label{l03}
Assume that the Markov chain induced by the generator $\ms L$ is
irreducible and that $\mu(x)>0$ for all $x\in V$. Then, there exists a
function $H\colon V\to \bb R$ such that $\mu$ is stationary for
$\mf M_H \ms L$.
\end{corollary}

\begin{corollary}
\label{l33}
Assume that the Markov chain induced by the generator $\ms L$ is
irreducible and that $\mu(x)>0$ for all $x\in V$. Then, $I(\mu)=0$ if
and only if $\mu$ is the stationary state.
\end{corollary}

\begin{proof}
Assume that $I(\mu)=0$. Then, since $a\, e^a \,-\, e^a \,+\, 1 \ge 0$,
each term in the sum \eqref{45} vanishes, and $H(y)=H(x)$ if
$\mu(x)\, R(x,y)>0$. As the Markov chain is irreducible, $H$ is
constant. To complete the argument, it remains to recall that, by
Lemma \ref{l10}, $\mu$ is a stationary state for the chain induced by
the generator $\mf M_H \ms L = \ms L$.

Conversely, suppose that $\mu$ is the stationary state. Then, $\mu$ is
a stationary state for the chain induced by the generator
$\mf M_H \ms L$ for $H=0$. By Lemma \ref{l10}, $I(\mu)= J_0(\mu)=0$,
as claimed.
\end{proof}

\subsection*{Reducible Markov chains}

In this subsection, we derive a formula for $I(\mu)$ in the case where
the process $X_t$ is reducible or the support of $\mu$ a proper subset
of $V$.

Denote by $\mf R_{\ms A} \ms L$, $\ms A$ a proper subset of $V$ which
is not a singleton, the generator of the Markov chain $X_t$ reflected
at $\ms A$. This is the $\ms A$-valued Markov chain which jumps from
$x\in \ms A$ to $y\in \ms A$ at rate $R(x,y)$. Its generator reads as
\begin{equation}
\label{54}
{\color{blue} [\, (\mf R_{\ms A} \ms L)\, f\,]\, (x)}
\;=\; \sum_{y\in \ms A} R(x,y) \, [\, f(y) -
f(x)\,]\;, \quad x\in \ms A \;. 
\end{equation}
Clearly, this chain may be reducible even if the original one is
irreducible.

Fix a probability measure $\mu \in \ms P(V)$, and denote by $V_\mu$
its support, $\color{blue} V_\mu = \{x\in V : \mu(x)>0\}$, and by
$\color{blue} X^\mu_t$ the Markov chain reflected at $V_\mu$.  Mind that
we do not assume $V_\mu$ to be a proper subset of $V$. 

The formula for $I(\mu)$ relies on the construction of a directed
graph without directed loops. Denote by by
$\color{blue} \ms Q_1, \dots, \ms Q_\ell$ the equivalence classes of
the chain $X^\mu_t$. These classes form the set of vertices of the
directed graph.  Draw a directed arrow from $\ms Q_a$ to $\ms Q_b$ if
there exists $x\in \ms Q_a$ and $y\in \ms Q_b$ such that
$R(x,y)>0$. Denote the set of directed edges by $\bb A$ and the graph
by $\color{blue} \bb G = (\mb Q, \bb A)$, where $\mb Q$ is the set
$\{\ms Q_1, \dots, \ms Q_\ell\}$ of vertices.

A path in the graph $\bb G$ is a sequence vertices
$(\ms Q_{a_j} : 0\le j\le m)$, such that there is a directed arrow
from $\ms Q_{a_j}$ to $\ms Q_{a_{j+1}}$ for $0\le j<m$.  This directed
graph has no directed loops because the existence of a directed loop
would contradict the definition of the sets $\ms Q_a$ as equivalent
classes. (Mind that undirected loops might exist).

Let $\color{blue} \ms C_1, \dots, \ms C_p$ be the closed irreducible
classes and $\color{blue} \ms T_1, \dots, \ms T_q$ be the transient
ones, so that $p+q=\ell$. Since the sets $\ms C_j$ are closed
irreducible classes, these sets are not the tail of a directed edge in
the graph. On the other hand, as the elements of $\ms T_i$ are
transient for the chain $X^\mu_t$, there is a path
$(\ms T_i= \ms T_{a_0}, \dots, \ms T_{a_m-1}, \ms C_j)$ from $\ms T_i$
to some irreducible class $\ms C_j$.

Fix a transient class $\ms T_i$.  Denote by
$\color{blue} \mb D(\ms T_i)$ the length of the longest path from
$\ms T_i$ to a closed irreducible class. The function $\mb D$ is well
defined because (a) the set of vertices is finite, (b) there is at
least a path, (c) there are no directed loops in the graph.

Fix $a$, $b$ such that there is a directed arrow from $\ms T_a$ to
$\ms T_b$. Then, 
\begin{equation}
\label{f10}
\mb D(\ms T_a) \;\ge\; \mb D(\ms T_b) \;+\; 1\;.
\end{equation}
Indeed, it is enough to consider the longest path from $\ms T_b$ to
the irreducible classes. $\ms T_a$ does not belong to the path because
there are no directed loops. By adding $\ms T_a$ at the beginning of
the path from $\ms T_b$ to the irreducible classes, we obtain a path
from $\ms T_a$ to the irreducible classes of length
$\mb D(\ms T_b) + 1$, proving \eqref{f10}.

Setting $\mb D(\ms C_j)=0$ for all $1\le j\le p$, we may extend
\eqref{f10} to the closed irreducible classes. Fix $a$, $b$ such that
there is a directed arrow from $\ms T_a$ to $\ms C_b$. Then,
\begin{equation}
\label{f10b}
\mb D(\ms T_a) \;\ge\; \mb D(\ms C_b) \;+\; 1
\end{equation}
because $\mb D(\ms T_a) \ge 1$.  Finally, we may lift the function
$\mb D$ to $V_\mu$ by setting $\color{blue} \mb D(x) = \mb D(\ms Q_a)$
for all $x\in \ms Q_a$. 

Recall from \eqref{54} that we represent by $\mf R_{\ms A} \ms L$ the
generator of the process $X_t$ reflected at $\ms A$.  Assume that the
chain induced by the generator $\mf R_{\ms A} \ms L$ is
irreducible. Let
$\color{blue} I_{\mf R_{\ms A} \ms L}\colon\ms P(\ms A) \to \bb R_+$
the functional given by
\begin{equation*}
I_{\mf R_{\ms A} \ms L}  (\mu) \;: =\;
\sup_{H} \, -\, \int_{\ms A} e^{-H} \,
\big[\, (\, \mf R_{\ms A} \ms L \,) \, e^H\,\big]
\, d\mu \;,
\end{equation*}
where the supremum is carried over all functions $H\colon\ms A \to \bb R$.

Denote by $\ms D_a$, $1\le a\le m$ the equivalent classes of the chain
$X^\mu_t$ with at least two elements. Note that $\mu(\ms D_a)>0$ for
all $a$ and that $m\le \ell$. Let $\mu_{\ms A}$, $\ms A \subset V$
such that $\mu(\ms A)>0$, be the measure $\mu$ conditioned to $\ms A$:
\begin{equation}
\label{57}
{\color{blue} \mu_{\ms A}(x)} \;:=\; \frac{\mu(x)}{\mu(\ms A)}\;, \quad
x\,\in \,  \ms A\;.
\end{equation}
Let $K \colon \ms P (V) \to \bb R_+$ the functional given by
\begin{equation}
\label{05}
\begin{aligned}
{\color{blue} K (\mu)} \, & =\,
\sum_{a=1}^{m}  \mu(\ms D_a)\,
I_{\mf R_{\ms D_a} \ms L} (\mu_{\ms D_a} ) 
\; +\;
\sum_{x\in V_\mu} \sum_{y\not\in V_\mu} \mu(x)\, R(x,y) \\
\; &+\;
\sum_{a=1}^\ell \sum_{b\neq a}
\sum_{x\in \ms Q_a} \sum_{y\in \ms Q_b} \mu(x)\, R(x,y)  \;.
\end{aligned}
\end{equation}
In this formula, since closed irreducible equivalent classes are not
the tail of any directed edge, we may restrict the sum over $a$ to
transient equivalent classes. Moreover, if $R(x,y)>0$ for some
$x\in \ms Q_a$, $y\in\ms Q_b$, then $R(z,w)=0$ for all $z\in \ms Q_b$,
$w\in\ms Q_a$. Hence, in the last sum, each pair $(a,b)$ is counted
only once.

In view of Lemma \ref{l02}, we may rewrite \eqref{05} as
\begin{equation}
\label{05b}
K (\mu) \, =\, -\, \sum_{a=1}^{m}
\sum_{x\in \ms D_a} \sum_{y\in \ms D_a\setminus \{x\}} \mu(x) \, R_{H_a}(x,y) 
\; +\;
\sum_{x\in V} \sum_{y\in V\setminus \{x\}}  \mu(x)\, R(x,y) \;,
\end{equation}
where $H_a: \ms D_a \to \bb R$ is the function (unique up to an
additive constant) which turns the measure $\mu$ conditioned to
$\ms D_a$ stationary for the chain induced by
$\mf M_{H_a} \mf R_{\ms D_a}\, \ms L$ (the generator
$\mf R_{\ms D_a}\, \ms L$ tilted by $H_a$).

\begin{lemma}
\label{l04}
For all $\mu \in \ms P(V)$
\begin{equation*}
I (\mu) \, =\, K(\mu)   \;.
\end{equation*}
\end{lemma}

\begin{proof}
We first prove that $I(\mu) \le K(\mu)$.  Fix $H\colon V \to \bb
R$. By definition of $V_\mu$,
\begin{equation*}
\int_V e^{-H} \ms L e^H \; d\mu \;=\;
\sum_{x\in V_\mu, y\in V} \mu(x)\, R(x,y)\, \big[\,
e^{H(y)-H(x)}-1\,\big]\;. 
\end{equation*}
The right-hand side can be rewritten as
\begin{equation}
\label{02}
\sum_{x\in V_\mu, y\not\in V_\mu}  c_H(x,y)
\;+\; \sum_{a=1}^{m} \sum_{x\in \ms D_a , y\in \ms D_a} 
c_H(x,y)   \;+\;
\sum_{a=1}^\ell \sum_{b\neq a}
\sum_{x\in \ms Q_a} \sum_{ y\in \ms Q_b} c_H(x,y) \;, 
\end{equation}
where $c_H(x,y) = \mu(x)\, R(x,y)\, [\,\exp \{ H(y)-H(x)\}-1\,]$.  

As $c_H(x,y) \ge - \, \mu(x)\, R(x,y)$, the sum of the first and third
terms of the previous displayed equation are bounded below by
\begin{equation*}
-\, \sum_{x\in V_\mu, y\not\in V_\mu} \mu(x)\, R(x,y)
\;-\;
\sum_{a=1}^\ell \sum_{b\neq a}
\sum_{x\in \ms Q_a} \sum_{ y\in \ms Q_b}
\mu(x)\, R(x,y)  \;.
\end{equation*}
The second term of that formula is bounded below by
\begin{equation*}
\sum_{a=1}^{m} \inf_{G} \sum_{x\in \ms D_a, y\in \ms D_a}
c_G(x,y)   \;=\; - \, \sum_{a=1}^{m}
\mu (\ms D_a)\, I_{\mf R_{\ms D_a} \ms L} (\mu_{\ms D_a}) \;.
\end{equation*}

Up to this point we proved that
\begin{equation*}
\int_V e^{-H} \ms L e^H \; d\mu \;\ge \; -\, K(\mu)
\end{equation*}
for all $H\colon V \to \bb R$. Multiplying by $-1$ and optimising over $H$
yields that $I(\mu) \le K(\mu)$.

We turn to the converse inequality.  For each set $\ms D_a$ the Markov
chain induced by the generator $\mf R_{\ms D_a} \ms L$ is irreducible and
$\mu(x)>0$ for all $x\in \ms D_a$. Hence, by Lemma \ref{l02}, there
exists a function $G_a\colon \ms D_a \to \bb R$ which solves the
variational problem
\begin{equation}
\label{04}
I_{\mf R_{\ms D_a} \ms L} (\mu_{\ms D_a}) \;=\; 
\sup_H \, -\, \int_{\ms D_a}  e^{-H}
\, (\mf R_{\ms D_a} \ms L)\,  e^H \; d\mu_{\ms D_a} \;= \;
-\, \int_{\ms D_a}  e^{-G_a} \, (\mf R_{\ms D_a} \ms L)\,  e^{G_a}
\; d\mu_{\ms D_a} \;,
\end{equation}
where the supremum is carried over all functions
$H\colon \ms D_a \to \bb R$. 

Recall the definition of the function $\mb D \colon V_\mu \to \bb R$
introduced above \eqref{f10}.  Define the sequence of functions
$H_n \colon V \to \bb R$ by
\begin{equation}
\label{29}
H_n(x) \;=\;
\begin{cases}
\displaystyle
-\, n &  x\not\in V_\mu\;,
\\
\displaystyle
G_a(x) \;+\; n\, \mb D(x)
& x\in \ms D_a\;,
\\
\displaystyle
n\, \mb D(x)
& x\in V_\mu \setminus \cup_{1\le a\le m} \ms D_a \;.
\end{cases}
\end{equation}
By definition of the functional $I$ and since $V_\mu$ stands for the
support of $\mu$,
\begin{equation*}
I(\mu) \;\ge\; -\, \liminf_{n\to\infty}
\int_{V_\mu}  e^{-H_n} \ms L e^{H_n} \; d\mu\;.
\end{equation*}
For a fixed $n$ the previous integral is equal to the sum in
\eqref{02} with $H_n$ replacing $H$. By definition of $H_n$, as
$n\to\infty$, the first term in \eqref{02} converges to
\begin{equation*}
-\, \sum_{x\in V_\mu, y\not\in V_\mu} \mu(x)\, R(x,y)  \;.
\end{equation*}
Since $\mb D(x) = \mb D(y)$ for elements $x$, $y$ in the same
equivalent class $\ms D_a$, by \eqref{04}, for every $n\ge 1$, the
second term in \eqref{02} is equal to
\begin{equation*}
\begin{aligned}
\sum_{a=1}^{m} \sum_{x\in \ms D_a, y\in \ms D_a} c_{G_a}(x,y)
\; & =\; \sum_{a=1}^{m} \mu (\ms D_a) \,
\int_{\ms D_a}  e^{-G_a} \, (\mf R_{\ms D_a} \ms L)\, 
e^{G_a} \; d\mu_{\ms D_a} \\
\;& =\; -\, \sum_{a=1}^{m}
\mu (\ms D_a) \, I_{\mf R_{\ms D_a} \ms L} (\mu_{\ms D_a})\;.
\end{aligned}
\end{equation*}
Finally, to estimate the third term in \eqref{02}, fix $x\in \ms Q_a$,
$y\in\ms Q_b$ such that $R(x,y)>0$.  By \eqref{f10}, \eqref{f10b},
$\mb D(x) \ge \mb D(y)+1$. Thus, $H_n(y) - H_n(x) \le -n + C_0$ for
some finite constant $C_0$ independent of $n$, and, as $n\to\infty$,
the second term in \eqref{02} converges to
\begin{equation*}
-\, \sum_{a=1}^\ell \sum_{b\neq a}
\sum_{x\in \ms Q_a} \sum_{ y\in \ms Q_b} \mu(x)\, R(x,y)
\end{equation*}
Collecting all previous estimates yields that
\begin{equation*}
I (\mu) \;\ge\;  -\, \liminf_{n\to\infty}
\int_{V}  e^{-H_n} \ms L e^{H_n} \; d\mu
\;=\;  K(\mu)\;,
\end{equation*}
which completes the proof of the lemma.
\end{proof}

\begin{lemma}
\label{l32}
A measure $\mu\in \ms P(V)$ is a stationary state of the Markov chain
$X_t$ if and only if $I(\mu)=0$. 
\end{lemma}

\begin{proof}
Assume that $I(\mu)=0$. Then, by Lemma \ref{l04}, all terms on the
right-hand side of \eqref{05} vanish. As the second and third terms
vanish the support of $\mu$ consists of the union of closed
irreducible sets of the Markov chain. Since the first term vanishes,
by Corollary \ref{l33}, $\mu$ restricted to these irreducible classes
is a stationary state. Hence, $\mu$ is a convex combination of the
stationary states, and, hence, a stationary state.

Suppose that $\mu$ is a stationary state. Then, its support is the
union of closed irreducible classes. Therefore, the second and third
terms on the right-hand side of \eqref{05} vanish. Fix a closed
irreducible class of the chain contained in the support of $\mu$. The
restriction of $\mu$ to this set is strictly positive. Hence, by
Corollary \ref{l33}, the first term on the right-hand side of
\eqref{05} also vanishes. This completes the proof of the lemma.
\end{proof}

\subsection*{Degenerate generators}

In this subsection, we consider generators whose holding rates might
vanish. Let $\color{blue} V_0 = \{x: \lambda(x)>0\}$ and keep in mind
that $V_0$ may be a proper subset of $V$.

Denote by $\ms L_0$ the generator $\ms L$ restricted to $V_0$:
\begin{equation*}
{ \color{blue} (\ms L_0 f)(x)}
\;=\; \sum_{y\in V_0} R(x,y)\, \{\, f(y)\,-\, f(x)\,\}\;, \quad
f\colon V_0\to \bb R\;.
\end{equation*}
Fix a measure $\mu\in \ms P(V)$. Clearly, if $\mu(V_0)=0$, then for
all $H\colon V\to \bb R$, $J_H(\mu) =0$ and $I(\mu)=0$. The next lemma
covers the case where $\mu(V_0)>0$. To stress the dependence of the
functional $J_H$ introduced at \eqref{37} on the generator $\ms L$,
denote it below by $\color{blue} J_{\ms L, H}$. For a function
$G\colon V_0\to \bb R$, let $J_{\ms L_0, G}
\colon \ms P(V_0) \to \bb R$ be the functional given by
\begin{equation*}
{\color{blue} J_{\ms L_0, G} (\nu)} \;:=\;
-\, \int_{V_0} e^{-G} \ms L_0 e^G\, d\nu\;.
\end{equation*}
Let $\color{blue} I_{\ms L_0} \colon \ms P(V_0) \to \bb R$ be the
functional defined by \eqref{19} with $J_H$ replaced by
$J_{\ms L_0, G}$ and where the supremum is carried over all functions
$G\colon V_0\to \bb R$.

\begin{lemma}
\label{l14}
For all measures $\mu\in \ms P(V)$ such that $\mu(V_0)>0$,
\begin{equation*}
J_{\ms L, H} (\mu) \;=\; \mu(V_0) \,
\Big\{\, J_{\ms L_0, H_{V_0}} (\mu_{V_0}) \,-\,
\sum_{x\in V_0} \sum_{y\in V\setminus V_0}
\mu_{V_0}(x) \, R(x,y)\, \big[e^{H(y)-H(x)} -1 \,\big]\,\Big\}
\;, 
\end{equation*}
where $H_{V_0}\colon V_0\to \bb R$ stands for the restriction of $H$
to $V_0$: $H_{V_0} (x) = H(x)$, $x\in V_0$. In particular,
\begin{equation*}
I_{\ms L}(\mu) \;=\;  \mu(V_0) \, I_{\ms L_0}(\mu_{V_0})
\;+\; \sum_{x\in V_0} \sum_{y\in V\setminus V_0}
\mu(x) \, R(x,y)\;.
\end{equation*}
\end{lemma}

\begin{proof}
The first assertion of the lemma is a simple identity.  The proof of
the second one is similar to the one of Lemma \ref{l04}.  We argue as
in this lemma to show that $I_{\ms L}(\mu)$ is bounded by the
right-hand side of the identity. The converse inequality is obtained
by observing that to optimize $J_{\ms L, H} (\mu)$ it is convenient to
set $H(y)=-\infty$ for $y\in V\setminus V_0$. More precisely, one
proceeds just as in the proof of Lemma \ref{l04} to obtain an optimal
sequence $H_n$ defined in $V_0$ and then extend it to $V\setminus V_0$
in such a way that $H_n (y) - H_n(x) \to -\infty$ for all
$y\in V\setminus V_0$, $x\in V_0$.
\end{proof}

The previous lemma allows us to restrict our attention to non-singular
generators. This is the content of the next result.

\begin{corollary}
\label{l38}
For all measures $\mu\in \ms P(V)$ such that $\mu(V_0)>0$,
\begin{equation*}
I_{\ms L} (\mu) \;=\; 
-\, \sum_{a=1}^{m}
\sum_{x\in \ms D_a} \sum_{y\in \ms D_a\setminus \{x\}} \mu(x) \, R_{H_a}(x,y) 
\; +\;
\sum_{x\in V} \sum_{y\in V\setminus \{x\}}  \mu(x)\, R(x,y) \;,
\end{equation*}
where $\ms D_a$, $1\le a\le m$, represent the equivalent classes of
the reflected chain $X^\mu_t$ with at least two elements, and
$H_a: \ms D_a \to \bb R$ the function (unique up to an
additive constant) which turns the measure $\mu$ conditioned to
$\ms D_a$ stationary for the chain induced by
$\mf M_{H_a} \mf R_{\ms D_a}\, \ms L$ (the generator
$\mf R_{\ms D_a}\, \ms L$ tilted by $H_a$).
\end{corollary}

\begin{proof}
By Lemma \ref{l14}, $I_{\ms L}(\mu)$ is the sum of two terms. Consider
$\mu(V_0) \, I_{\ms L_0}(\mu_{V_0})$. By Lemma \ref{l04},
$I_{\ms L_0}(\mu_{V_0}) = K_{\ms L_0}(\mu_{V_0})$, where $K_{\ms L_0}$
is given by equation \eqref{05b} with the set $V$ replaced by $V_0$ in
the second sum. Clearly, the equivalent classes of the reflected chain
$X^\mu_t$ with at least two elements for the generator $\ms L_0$
coincide with the ones for the generator $\ms L$. On the other hand
there is nor harm to replace in the second sum of \eqref{05b} the
condition $x\in V_0$ by $x\in V$ as $R(x,y)=0$ for all
$x\in V\setminus V_0$.  Hence,
\begin{equation*}
\mu(V_0) \, I_{\ms L_0}(\mu_{V_0}) \;=\;
-\, \sum_{a=1}^{m}
\sum_{x\in \ms D_a} \sum_{y\in \ms D_a\setminus \{x\}} \mu(x) \, R_{H_a}(x,y) 
\; +\;
\sum_{x\in V} \sum_{y\in V_0\setminus \{x\}}  \mu(x)\, R(x,y)\;.
\end{equation*}
To complete the proof of the corollary, it remains to recall the
formula for $I_{\ms L}(\mu)$ presented in Lemma \ref{l14}.
\end{proof}

\section{Convergence of level 2 rate functionals}
\label{sec-a2}

In this section, we present some general results on the convergence of
level 2 large deviations rate functionals.  The first result asserts
that the rate functionals converge provided the jump rates converge.

Denote by $\color{blue} \ms L_n$, $n\ge 1$, the generator of a
$V$-valued continuous-time Markov chain whose jump rates are
represented by $\color{blue} R_n(\,\cdot\,,\,\cdot\,)$.  Let
$\color{blue} I_n\colon \ms P(V) \to \bb R_+$ be the occupation time
large deviations rate functional associated to the generator
$\ms L_n$. This is the functional defined by formula \eqref{19} with
the rates $R_n$ replacing $R$. We first consider the case where
$\ms L$ is irreducible.

\begin{lemma}
\label{l27}
Suppose that the Markov chains induced by the generators $\ms L_n$,
$n\ge 1$, and $\ms L$ are irreducible, and that
$R_n(x,y) \to R(x,y) \in \bb R_+$ for all $y\not = x\in V$. 
Then, $I_n(\mu) \to I(\mu)$ for all $\mu\in \ms P(V)$ such that
$\mu(x)>0$ for all $x\in V$. Here, $I$ represents the rate functional
associated to the jump rates $R$.
\end{lemma}

\begin{proof}
Fix $H\colon V\to \bb R$ and denote by
$\color{blue} J^{(n)}_H: \ms P(V) \to \bb R$ the functional given by
\eqref{37} with $R$ replaced by $R_n$. Then, as $R_n(x,y) \to R(x,y)$,
for all $H\colon V\to \bb R$
\begin{equation*}
I_n(\mu) \;\ge\; J^{(n)}_H (\mu) \;\to\; J_H (\mu)\;.
\end{equation*}
Maximizing over $H$ yields that
$\liminf_{n\to\infty} I_n(\mu) \,\ge\, I(\mu)$. Note that we did not
use the irreducibility of $\ms L$ in this part of the proof.

Conversely, since the Markov chain induced by the generator $\ms L_n$
is irreducible and the support of $\mu$ is the set $V$, by Lemmata
\ref{l02} and \ref{l10}, $I_n(\mu) = J^{(n)}_{H_n} (\mu)$, where $H_n$
is the function which turns $\mu$ the stationary stated for the tilted
generator $\mf M_{H_n}\,\ms L_n$.

Denote by $\ms E$ the oriented edges $(x,y)\in V \times V$ such that
$R(x,y)>0$. Let $a= \min\{ \mu(x) \, R(x,y) : (x,y)\in \ms E\} >
0$. As $R_n$ converges to $R$, there exists $n_0>0$ such that
$\min\{ \mu(x) \, R_n(x,y) : (x,y)\in \ms E\} \ge a/2$. Since the
Markov chain induced by the rates $R(x,y)$ is irreducible and the
rates $R_n$ converge to $R$, by \eqref{38} and Remark \ref{rm1}, there
exists a finite constant $C_0$ such that
\begin{equation}
\label{39}
\max_{y,x\in V} |\, H_n (y) - H_n(x)\,|\, \le\, C_0
\end{equation}
for all $n\ge n_0$. 

Therefore, there exist functions $G\colon V \to \bb R$ and a
subsequence $n'$ such that $H_{n'} (y) - H_{n'}(x) \to G(y) - G(x)$
for all $x$, $y\in V$. Hence, through this subsequence
$J^{(n)}_{H_n} (\mu)$ converges to $J_{H} (\mu) \le I(\mu)$.
This proves that $\limsup_n I_n(\mu) \le I(\mu)$ and completes the
proof of the lemma.
\end{proof}

We now remove the assumption that $\ms L$ is irreducible.

\begin{lemma}
\label{l26}
Suppose that the Markov chain induced by the generator $\ms L_n$ is
irreducible for all $n\ge 1$ and that
$R_n(x,y) \to R(x,y) \in \bb R_+$ for all $y\not = x\in V$. Then,
$I_n(\mu) \to I(\mu)$ for all $\mu\in \ms P(V)$ such that $\mu(x)>0$
for all $x\in V$. 
\end{lemma}

\begin{proof}
In Lemma \ref{l27}, we proved that
$\liminf_{n\to\infty} I_n(\mu) \,\ge\, I(\mu)$.  Conversely, since the
Markov chain induced by the generator $\ms L_n$ is irreducible and the
support of $\mu$ is the set $V$, by Lemmata \ref{l02} and \ref{l10},
$I_n(\mu) = J^{(n)}_{H_n} (\mu)$, where $H_n$ is the function which
turns $\mu$ the stationary stated for the tilted generator
$\mf M_{H_n}\,\ms L_n$.

Denote by $\mf Q_a$, $1\le a\le \ell_\mu$ the equivalent classes of
the generator $\ms L$, and by $\mf D_a$, $1\le a\le m_\mu$ the ones
with at least two elements. By definition of $J^{(n)}_{H_n} (\mu)$,
\begin{equation*}
\begin{aligned}
& J^{(n)}_{H_n} (\mu) \; \le \;
\sum_{a=1}^{m_\mu} \mu (\mf D_a)\,
J^{(n)}_{\mf D_a, H_n} (\mu_{\mf D_a})
\; +\;
\sum_{a=1}^{\ell_\mu} \sum_{b\neq a}
\sum_{x\in \mf Q_a} \sum_{y\in \mf Q_b}
\mu(x) \, R_{n}(x,y)\;, \\
&\quad \text{where}\;\;
J^{(n)}_{\mf D_a, H_n} (\nu) \;=\; 
\sum_{x\in \mf D_a} \sum_{y\in \mf D_a\setminus \{x\}}
\nu(x) \, R_{n}(x,y)\, \big[\, 1 \,-\, e^{H_n(y) - H_n(x)}\,\big]
\end{aligned}
\end{equation*}
As $n\to\infty$, the second term converges to the same sum with $R$ in
place of $R_n$. We turn to the first term.

Fix $1\le a\le m_\mu$. By the arguments presented in the proof of
Lemma \ref{l27}, \eqref{39} holds provided the maximum is carried over
$x$, $y\in \mf D_a$. Therefore, by the end of the proof of that lemma,
$\limsup_n J^{(n)}_{\mf D_a, H_n} (\mu_{\mf D_a}) \le I_{\mf R_{\mf
D_a} \ms L} (\mu_{\mf D_a})$. Recollecting the previous estimates and
recalling Lemma \ref{l04} and definition \eqref{05} yields that
$\limsup_n I_n (\mu) \le I (\mu)$. This completes the proof of the
lemma.
\end{proof}

\subsection*{$\Gamma$-convergence}

In this subsection we present a result on $\Gamma$-convergence used in
the article. Recall from Section \ref{sec1} the definition of
$\Gamma$-convergence.  Fix a Polish space $\mc X$ and a functional
$U \colon \mc X \to [0,+\infty]$.

\begin{definition}
\label{l25}
A subset $\mc X_0$ of $\mc X$ is said to be $U$-dense if for every
$x\in \mc X$ such that $U(x)<\infty$, there exists a sequence
$(x_k: k \ge 1)$ such that $x_k\in \mc X_0$, $x_k \to x$ and
$U(x_k) \to U(x)$.
\end{definition}

\begin{lemma}
\label{l24}
Let $\mc X_0$ be a $U$-dense subset of $\mc X$. To show that $U$ is a
$\Gamma$-limsup for the sequence $U_n$, it is enough to show that for
every $x\in \mc X_0$, there exists a sequence $(x_n:n\ge 1)$ such that
$x_n\to x$ and \eqref{30} holds.
\end{lemma}

\begin{proof}
Assume that for each $x\in \mc X_0$, there exists a sequence
$(x_n:n\ge 1)$ such that $x_n\to x$ and \eqref{30} holds.

Fix $x \in \mc X$.  We have to show that \eqref{30} holds for some
sequence $x_n \in \mc X$ which converges to $x$. If $U(x) = \infty$,
there is nothing to prove. Assume, therefore, that $U(x) < \infty$.
As $\mc X_0$ is $U$-dense, there exists a sequence $(x^{(k)}: k\ge 1)$
such that $x^{(k)} \in \mc X_0$, $x^{(k)} \to x$, and
$U (x^{(k)}) \to U(x)$.

Since the result holds for elements of $\mc X_0$, for each $k\ge 1$,
there exists a sequence $(x^{(k)}_n : n\ge 1)$ such that
$x^{(k)}_n \to x^{(k)}$,
$\limsup_{n\to\infty} U_n (x^{(k)}_n) \le U (x^{(k)})$. At this point,
a classical diagonal argument permits to construct a sequence $x_n$
such that $x_n \to x$, $\limsup_{n\to\infty} U_n (x_n) \le U (x)$, as
claimed.
\end{proof}

We return to the context of the article and assume that $\mc X = \ms
P(V)$. 

\begin{lemma}
\label{l15}
Let $\ms P_+$ be the subset of $\ms P(V)$ formed by the measures whose
support is $V$:
$\ms P_+ = \{\mu \in \ms P(V): \mu(x)>0 \; \forall\, x\in V\}$. The set
$\ms P_+$ is $I$-dense.
\end{lemma}

\begin{proof}
Fix $\mu\in \ms P(V)$ and let $\nu$ be the uniform probability measure
on $V$. Set $\mu_n \,=\, [1-(1/n)] \, \mu + (1/n)\,\nu$.  Clearly,
$\mu_n \in \ms P(V)$ and $\mu_n\to \mu$. It remains to show that
$I(\mu_n) \to I(\mu)$.

Recall from Lemma \ref{l01} the properties of the functional $I$.  By
the lower-semicontinuity of $I$, $I(\mu) \le \liminf_n I(\mu_n)$. By
convexity, $I(\mu_n) \le [1-(1/n)] \, I(\mu) \,+\, (1/n)\,I(\nu)$.
Since $I(\nu) \le |V|^{-1}\, \sum_{x\in V} \lambda (x) < \infty$,
$\limsup_n I(\mu_n) \le I(\mu)$, as claimed.
\end{proof}

\section{Trace process and level 2 rate functionals}
\label{sec-a3}

We examine in this section the effect of reducing the state space, by
taking the trace of the process, on the large deviations rate
functional. We first recall the definition of the trace process and
some of its properties.

Denote by $T^{\ms A}(t)$, $\ms A \subsetneq V$, the total time the
process $X_t$ spends in $\ms A$ in the time-interval $[0,t]$:
\begin{equation*}
T^{\ms A}(t)\;=\;\int_{0}^{t}\,\chi_{\ms A}(X_s)\,ds\;,
\end{equation*}
where, recall, $\chi_{\ms A}$ represents the indicator function of the set
$\ms A$. Denote by $S^{\ms A}(t)$ the generalized inverse of $T^{\ms A}(t)$:
\begin{equation*}
S^{\ms A}(t)\;=\;\sup\{\,s\ge0\,:\,T^{\ms A}(s)\le t\,\}\;.
\end{equation*}

The trace of $X_t$ on $\ms A$, denoted by
$\color{blue}(X^{\ms A}_t : t \ge 0)$, is defined by
\begin{equation}
\label{100}
X^{\ms A}_t\;=\; X_{S^{\ms A}(t)} \;;\;\;\;t\ge0\;.
\end{equation}
By Propositions 6.1 and 6.3 in \cite{bl2}, the trace process is an
irreducible, $\ms A$-valued continuous-time Markov chain, obtained by
turning off the clock when the process $X_t$ visits the set
$\ms A^{c}$, that is, by deleting all excursions to $\ms A^{c}$. For
this reason, it is called the trace process of $X_t$ on $\ms A$.  For
a $V$-valued Markov chain generator $\ms L$ and a proper subset
$\ms A$ of $V$, denote by $\color{blue} \mf T_{\ms A} \ms L$ the
generator of the trace process on $\ms A$.

Denote by $\color{blue} R^{T}_{\ms A}(\,\cdot\,,\,\cdot\,)$ the jump
rates of the trace process. Suppose that $\ms A=V\setminus \{z\}$. The
first equation after \cite[Corollary 6.2]{bl2} asserts that for all
$x$, $y\in \ms A$,
\begin{equation}
\label{53}
R^{T}_{V\setminus \{z\}} (x,y)
\;=\; R(x,y) \;+\; R(x,z) \, p(z,y)\;,
\end{equation}
where $p(\,\cdot\,,\,\cdot\,)$ represents the jump probability of the
chain $X_t$. In particular, $R^{T}_{V\setminus \{z\}} (x,y)
\ge R(x,y)$. Iterating this procedure yields that
\begin{equation}
\label{52}
R^{T}_{\ms A}(x,y) \;\ge\; R(x,y)
\end{equation}
for all $\ms A\subset V$, $x \neq y\in \ms A$.

We turn to the rate functional.  To simplify certain formulae, we
represent the rate function as
\begin{equation*}
I(\mu) \;=\; \sup_u \,-\, \int_V \frac{\ms L u}{u}\; d\mu\;,
\end{equation*}
where the supremum is carried out over all strictly positive functions
$u\colon V\to (0,\infty)$.  In consequence, in this section,
$\mf M_u \ms L$ stands for the tilted generator given by
\begin{equation}
\label{55}
(\mf M_u \ms L f)(x) \;=\; \sum_{y\in V} R_u(x,y) \, [\, f(y)
\,-\, f(x)\,]\;, \quad 
R_u(x,y) \,=\, \frac{1}{u(x)} \, R(x,y)\, u(y) \;.
\end{equation}

Denote by $I^{\rm T}_{\ms A}$ the large deviations rate functional
associated to the trace generator $\mf T_{\ms A} \ms L$:
\begin{equation*}
{\color{blue} I^{\rm T}_{\ms A} (\mu)} \;:=\; \sup_{u}
\,-\, \int_{\ms A} 
\frac{ (\mf T_{\ms A} \ms L)  \, u}{u} \; d\mu\;,
\quad \mu\in\ms P(\ms A)\;,
\end{equation*}
where the sup is carried over all functions
$u\colon\ms A \to (0, \infty)$.

Fix $u\colon\ms A \to (0,\infty)$, and denote by
$\color{blue} \mf H\, u = \mf H_{\ms L}\, u$ the $\ms L$-harmonic
extension of $u$ to $V$, that is, the solution of the Poisson equation
\begin{equation}
\label{07}
\left\{
\begin{aligned}
& \ms L v \,=\, 0\;, \quad V\setminus \ms A\;, \\
& v \,=\, u\;, \quad \ms A\;.
\end{aligned}
\right.
\end{equation}

Next proposition is the main result of this section. 

\begin{proposition}
\label{l05}
Assume that the Markov chain induced by the generator $\ms L$ is
irreducible.  Fix $\mu\in \ms P(V)$ and assume that its support,
denoted by $\ms A$, is a proper subset of $V$. Then,
\begin{equation*}
I^{\rm T}_{\ms A} (\mu) \;\le\; I(\mu)\;.
\end{equation*}
Conversely, let $u\colon \ms A \to (0, \infty)$ such that
\begin{equation*}
I^{\rm T}_{\ms A} (\mu) \;=\; -\, \int_{\ms A}
\frac{(\mf T _{\ms A} \ms L) \, u}{u}  \; d\mu  \;.
\end{equation*}
Denote by $v= \mf H_{\ms A} u$ the harmonic extension of $u$ to $V$
given by the solution of \eqref{07}. Let $\nu$ be the stationary state
of the tilted generator $\mf M_v \ms L$. Then,
\begin{equation*}
I^{\rm T}_{\ms A} (\mu) \;=\; 
\frac{1}{\nu(\ms A)} \; I(\nu) \;.
\end{equation*}
\end{proposition}

\subsection*{Harmonic extension}

The harmonic extension has a stochastic representation.  Denote by
$H_{\ms A}$, $H^+_{\ms A}$, ${\ms A}\subset V$, the hitting and return
time of ${\ms A}$:
\begin{equation} 
\label{201}
{\color{blue} H_{\ms A} } \;: =\;
\inf \big \{t>0 : X_t \in {\ms A} \big\}\;,
\quad
{\color{blue} H^+_{\ms A}} \;: =\;
\inf \big \{t>\tau_1 : X_t \in {\ms A} \big\}\; ,  
\end{equation}
where $\tau_1$ represents the time of the first jump of the chain
$X_t$: $\color{blue} \tau_1 = \inf\{t>0 : X_t \not = X_0\}$. By the
strong Markov property, the solution of the Poisson equation can be
represented as
\begin{equation}
\label{71}
(\mf H u) \, (x) \,=\, \;
\mb E_x[\, u(X_{H_{\ms A}}) \,]\;, \quad x\in V \;.
\end{equation}
In particular,
\begin{equation*}
\min_{y\in\ms A} u(y)
\;\le\;  \min_{x\in V} \, (\mf H \, u)\,  (x)
\;\le\; \max_{x\in V} \, (\mf H \, u)\,  (x)
\;\le\; \max_{y\in\ms A} u(y)  \;.
\end{equation*}
Moreover, by \cite[Lemma A.1]{bgl2},
\begin{equation}
\label{09}
[\, (\mf T_{\ms A} \ms L) \, u \,] (x)
\;=\;  [\, \ms L (\mf H\,  u)\,] (x)\;, \quad  x\in \ms A\;. 
\end{equation}

\begin{lemma}
\label{l08}
Fix a subset $\ms A \subsetneq \ms B \subsetneq V$, and a function
$u\colon \ms A \to \bb R$. Then, 
\begin{equation*}
(\mf H_{\mf T_{\ms B} \ms L} \, u)\, (x) \;=\; 
(\mf H_{\ms L} \, u)\, (x)\;, \quad x\in\ms B\;.
\end{equation*}
\end{lemma}

This result asserts that the $\ms L$-harmonic extension of $u$ to $V$
coincides on the set $\ms B$ with the $(\mf T_{\ms B} \ms L)$-harmonic
extension of $u$ to $\ms B$. More precisely, denote by
$v=\mf H_{\ms L} u$ the $\ms L$-harmonic extension of $u$ to $V$ and
by $v_{\ms B} \colon \ms B \to \bb R$ its restriction to $\ms B$,
defined as $v_{\ms B} (x) = v(x)$, $x\in \ms B$. Lemma \ref{l08}
states that the function $v_{\ms B}$ is the
$(\mf T_{\ms B} \ms L)$-harmonic extension of $u$ to $\ms B$.  In
other words, that $v_{\ms B}$ solves equation \eqref{07} with $\ms B$,
$\mf T_{\ms B} \ms L$ replacing $V$, $\ms L$, respectively.

\begin{proof}[Proof of Lemma \ref{l08}]
Denote by $X^{\ms B}_t$ the trace of the Markov chain $X_t$ on $\ms B$
and by $H_{\ms A} (X^{\ms B})$ its hitting time of the set $\ms
A$. Clearly, starting from $x\in \ms B$,
$X^{\ms B}_{H_{\ms A} (X^{\ms B})} = X_{H_{\ms A}}$ almost surely, so
that
\begin{equation*}
\mb E_x[\, u(X^{\ms B}_{H_{\ms A} (X^{\ms B})}) \,] \;=\;
\mb E_x[\, u(X_{H_{\ms A}}) \,]
\end{equation*}
for all $x\in \ms B$. By \eqref{71}, the left-hand side of this
equation is $(\mf H_{\mf T_{\ms B} \ms L} \, u)\, (x)$, and the
right-hand side is $(\mf H_{\ms L} \, u)\, (x)$. This completes the
proof of the lemma.
\end{proof}

Enumerate the set $V\setminus \ms A$ as $\{x_1, \dots, x_p\}$, and let
$\ms A_0=V$, $\ms A_k = V \setminus \{x_1, \dots , x_k\}$,
$1\le k\le p$, so that $\ms A = \ms A_p$. Fix
$u\colon \ms A \to (0,\infty)$, and denote by $v$ its $\ms L$-harmonic
extension to $V$, given by \eqref{07}. Let
$v_k\colon \ms A_k \to (0,\infty)$ be the restriction of $v$ to
$\ms A_k$, $1\le k\le p$, and $v_0=v$.

\begin{corollary}
\label{l07}
For all $1\le k\le p$, 
\begin{equation*}
[\, (\mf T_{\ms A_{k-1}}\, \ms L)\, 
v_{k-1}\,]\,  (y) \;=\; 0 \;, \quad y\in \ms A_{k-1} \setminus \ms A\;.
\end{equation*}
In particular, $v_{k-1}$ is the $(\mf  T_{\ms A_{k-1}}\, \ms
L)$-harmonic extension of $v_j$ for $k\le j\le p$.
\end{corollary}

\begin{proof}
Fix $1\le k\le p$.  By Lemma \ref{l08}, $v_{k-1}$ is the
$(\mf T_{\ms A_{k-1}} \ms L)$-harmonic extension of $u$ to
$\ms A_{k-1}$. Hence, for all $y\in \ms A_{k-1} \setminus \ms A$,
$[\, (\mf T_{\ms A_{k-1}} \ms L) \, v_{k-1}\,] \, (y)=0$, as
claimed. By definition, for $k\le j\le p$, $v_{k-1}$ and $v_j$
coincide on $\ms A_j$, which proves the second assertion of the
corollary. 
\end{proof}

\subsection*{Tilted dynamics}

Fix $v\colon V \to (0,\infty)$, and recall the definition of the jump
rates $R_v$ introduced in \eqref{55}. Denote by $\lambda_v (x)$,
$\lambda (x)$ the holding times at $x$ associated to the rates
$R_v(\,\cdot\,,\,\cdot\,)$, $R(\,\cdot\,,\,\cdot\,)$, respectively.
Assume that $(\ms L v)(x)=0$. Then,
\begin{equation}
\label{06}
\lambda_v (x) \;=\; \lambda (x)
\end{equation}

Indeed, since $(\ms L v)(x) =0$,
\begin{equation*}
\sum_{y\in V} R(x, y) \, v(y)
\;=\; \sum_{y\in V} R(x, y) \, v(x)\;.
\end{equation*}
Hence,
\begin{equation*}
\lambda_v (x)
\;=\; \frac{1}{v(x)} \sum_{y\in V} R(x, y) \, v(y)
\;=\; \frac{1}{v(x)} \sum_{y\in V} R(x, y) \, v(x)
\;=\; \lambda (x)\;, 
\end{equation*}
as claimed.

\begin{lemma}
\label{l06}
Fix $u\colon \ms A \to (0,\infty)$. Let $v\colon V\to (0,\infty)$ be
its harmonic extension to $V$ defined by \eqref{07}. Then,
\begin{equation*}
\mf T_{\ms A} \, (\mf M_v \ms L)
\;=\;  \mf M_ u\, (\mf T_{\ms A}  \ms L)\;.
\end{equation*}
\end{lemma}

\begin{proof}
We first prove the lemma for $\ms A = V \setminus \{x_0\}$. Recall
that we denote by $R_v(x,y)$ the jump rates of the tilted generator
$\mf M_v \ms L$. Denote by $R_{v, \ms A}(x,y)$ the jump rates of the
trace generator $\mf T_{\ms A} \mf M_v \ms L$.  By \eqref{53}, for all
$x \neq y \in\ms A$,
\begin{equation*}
R_{v, \ms A}(x,y) \;=\; R_{v}(x,y) \;+\;
\frac{1}{\lambda_v (x_0)}\, R_{v}(x,x_0)\, R_{v}(x_0, y)\;, 
\end{equation*}
where, recall, $\lambda_v (x_0)$ stands for the holding time at $x_0$
associated to the rates $R_v(x,y)$. Since $v$ is harmonic at $x_0$, by
\eqref{06} and the definition of $R_v$, the previous expression is
equal to
\begin{equation*}
\frac{v(y)}{v(x)}\, \Big\{R (x,y) \;+\;
\frac{1}{\lambda (x_0)}\, R (x,x_0)\, R (x_0, y) \,\Big\}\;.
\end{equation*}
Since $v$ and $u$ coincide on $\ms A$, we may replace $v(x)$, $v(y)$
by $u(x)$, $u(y)$, respectively. Representing by $R_{\ms A}(z,w)$ the
jump rates associated to the generator $\mf T_{\ms A} \ms L$, by
\eqref{53} once more, the previous expression is equal to
\begin{equation*}
\frac{u(y)}{u(x)}\, R_{\ms A} (x,y) \;,
\end{equation*}
which proves the lemma in the case where $V\setminus \ms A$ is a
singleton.

We extend the result to arbitrary sets $\ms A$. Fix a set $\ms A$, and
recall the notation introduced above Corollary \ref{l07}. Fix
$1\le j \le p$. By Corollary \ref{l07}, $v_{k-1}$ is the
$(\mf T_{\ms A_{k-1}}\, \ms L)$-harmonic extension of $v_k$. Thus, by
the first part of the proof applied to the generator
$\mf T_{\ms A_{k-1}} \ms L$,
\begin{equation}
\label{10}
\mf T_{\ms A_k}\, \mf M_{v_{k-1}}\, \mf T_{\ms A_{k-1}} \ms L
\;=\; \mf M_{v_k}\, \mf T_{\ms A_k}\, \mf T_{\ms A_{k-1}}\, \ms L
\;=\; \mf M_{v_k}\, \mf T_{\ms A_k}\, \ms L
\end{equation}
because $\mf T_{\ms A_k}\, \mf T_{\ms A_{k-1}}\, \ms L = \mf T_{\ms
A_k}\, \ms L$. 

Since $\ms A_0=V$, $\mf T_{\ms A_0}\, \ms L = \ms L$, and \eqref{10}
for  $k=1$ states that
\begin{equation*}
\mf T_{\ms A_1}\, \mf M_{v_{0}}\, \ms L
\;=\; \mf M_{v_1}\, \mf T_{\ms A_1}\, \ms L \;.
\end{equation*}
Applying $\mf T_{\ms A_2}$ on both sides of this identity and then
\eqref{10} yields that
\begin{equation*}
\mf T_{\ms A_2}\, \mf M_{v_{0}}\, \ms L
\;=\; \mf M_{v_2}\, \mf T_{\ms A_2}\, \ms L \;.
\end{equation*}
Iterating this procedure completes the proof of the lemma as $v_0=v$,
$v_p=u$, $\ms A_0 = V$, $\ms A_p = \ms A$.
\end{proof}

\subsection*{Proof of Proposition \ref{l05}}

By \cite[Proposition 6.1]{bl2}, the Markov chain induced by the trace
generator $\mf T_{\ms A}\, \ms L$ is irreducible. Hence, since
$\mu(x)>0$ for all $x\in \ms A$, by Lemma \ref{l02}, there exists
$u \colon\ms A\to (0,\infty)$ such that
\begin{equation*}
I^{\rm T}_{\ms A} (\mu) \;=\; -\, \int_{\ms A}
\frac{(\mf T _{\ms A} \ms L) \, u}{u}  \; d\mu
\;.
\end{equation*}
Denote by $v$ the harmonic extension of $u$ to $V$ given by the
solution of \eqref{07}. By \eqref{09},
$(\ms L v)(x) = [\, (\mf T _{\ms A} \ms L) \, u\,](x)$ for all
$x\in \ms A$.  Hence, the right-hand side
of the previous displayed equation is equal to
\begin{equation*}
-\, \int_{\ms A} \frac{\ms L \, v}{v} \; d\mu
\;=\; -\, \int_{V} \frac{\ms L \, v}{v} \; d\mu \;\le\; I(\mu) \;. 
\end{equation*}
The identity follows from the fact that $\ms A$ is the support of
$\mu$ (or from the fact that $v$ is harmonic in $\ms A^c$), and the
inequality from the definition of the functional $I$.

We turn to the converse assertion. By definition of $u$, $v$ and
\cite[Lemma A.1]{bgl2},
\begin{equation*}
I^{\rm T}_{\ms A} (\mu) \;=\; -\, \int_{\ms A}
\frac{(\mf T _{\ms A} \ms L) \, u}{u}  \; d\mu
\;=\; -\, \int_{\ms A}
\frac{\ms L \, v}{v}  \; d\mu \;.
\end{equation*}

Recall that $\nu$ represents the stationary state of the tilted
generator $\mf M_v \ms L$.  By \cite[Proposition 6.3]{bl2}, the
measure $\nu$ conditioned to $\ms A$ is the stationary state of the
trace (the Markov chain induced by the generator
$\mf T_{\ms A} \mf M_v \ms L$). Since, by Lemma \ref{l06},
$\mf T_{\ms A} \mf M_v \ms L = \mf M_ u\mf T_{\ms A} \ms L$, $\nu$
conditioned to $\ms A$ is the stationary state of the chain associated
to $\mf M_ u\mf T_{\ms A} \ms L$. By Lemma \ref{l01}, $\mu$ is
stationary for $\mf M_ u\mf T_{\ms A} \ms L$ as well. Since the chain
$X_t$ is irreducible, so is the trace and the tilted trace. Thus, by
uniqueness, $\mu(\,\cdot\,) = \nu(\,\cdot\,)/\nu(\ms A)$, and the
right-hand side of the previous displayed equation can be written as
\begin{equation*}
-\, \frac{1}{\nu(\ms A)}\, \int_{\ms A} \frac{\ms L \, v}{v}  \; d\nu
\;=\; -\, \frac{1}{\nu(\ms A)}\, \int_{V} \frac{\ms L \, v}{v}  \; d\nu
\end{equation*}
because $v$ is harmonic on $V\setminus \ms A$.
Since $\nu$ is stationary for the tilted generator $\mf M_v \ms L$, by
Lemma \ref{l10}, the right-hand side is equal to $\nu(\ms A)^{-1} I(\nu)$.
This completes the proof of the lemma. \qed

\end{document}